\documentclass[10pt]{amsart}
\usepackage{amsmath, amsthm, amsfonts, amssymb}
\usepackage{mathrsfs}
\providecommand{\noopsort[1]{}}
\usepackage{bbm}
\usepackage{dsfont}
\usepackage{ifthen}
\numberwithin{equation}{section}
\usepackage{a4}
\usepackage[active]{srcltx}
\usepackage{verbatim}
\usepackage{hyperref}

\newtheorem{thm}{Theorem}[section]
\newtheorem{cor}[thm]{Corollary}
\newtheorem{prop}[thm]{Proposition}
\newtheorem{lem}[thm]{Lemma}

\theoremstyle{remark}
\newtheorem{rem}[thm]{Remark}

\newtheorem{example}[thm]{Example}

\theoremstyle{definition}

\renewcommand{\Re}{{\rm Re}\,}

\newcommand{\eps}{\varepsilon}

\newcommand{\one}{\mathbbm{1}}

\newcommand{\sol}{{\rm{sol}}}

\newcommand{\n}{\Vert}

\newcommand{\form}[3]{\ifthenelse{\equal{#2}{}}{\mbox{$ #1\Big[\, \cdot\,  , \, \cdot\,  \Big]$}}{
\mbox{$ #1\Big[ #2 , #3 \Big]$}}}
\newcommand{\qform}[2]{\ifthenelse{\equal{#2}{}}{\mbox{$ #1\Big[\cdot \Big]$}}{
\mbox{$ #1\Big[ #2 \Big]$}}}
\newcommand{\ip}[2]{\ifthenelse{\equal{#1}{}}{\mbox{$ \Big( \,\cdot\; \vline \; \cdot \, \Big) $}}{
\mbox{$ \Big( #1 \;  \vline \; #2 \Big)$}}}
\newcommand{\norm}[1]{\ifthenelse{\equal{#1}{}}{\mbox{$\|\cdot\|$}}{\mbox{$\| #1 \|$}}}
\newcommand{\betr}[1]{\ifthenelse{\equal{#1}{}}{\mbox{$|\cdot|$}}{\mbox{$| #1 |$}}}
\newcommand{\dual}[2]{\ifthenelse{\equal{#1}{}}{\mbox{$ \Big\langle \,\cdot\; , \; \cdot \, \Big\rangle $}}{
\mbox{$ \Big\langle #1 \;  , \; #2 \Big\rangle$}}}
\newcommand{\pdual}[2]{\ifthenelse{\equal{#1}{}}{\mbox{$ \Big[ \,\cdot\; , \; \cdot \, \Big] $}}{
\mbox{$ \Big[ #1 \;  , \; #2 \Big]$}}}
\newcommand{\bdual}[2]{\ifthenelse{\equal{#1}{}}{\mbox{$ \Big\langle \,\cdot\; , \; \cdot \, \Big\rangle_* $}}{
\mbox{$ \Big\langle #1 \;  , \; #2 \Big\rangle_*$}}}

\newcommand{\CR}{\mathds{R}}

\newcommand{\CN}{\mathds{N}}
\newcommand{\OCN}{\overline{\mathds{N}}}

\renewcommand{\P}{\mathbb{P}}

\newcommand{\A}{\mathcal{A}}

\newcommand{\bS}{{S}}

\newcommand{\cF}{\mathscr{F}}

\newcommand{\expect}{\mathbb{E}}
\newcommand{\cL}{\mathscr{L}}

\newcommand{\half}{\frac{1}{2}}

\newcommand{\inject}{\hookrightarrow}

\def\N{{\mathds N}}

\def\R{{\mathds R}}

\newcommand{\E}{{\mathbb E}}
\renewcommand{\P}{{\mathbb P}}
\newcommand{\F}{{\mathscr F}}

\newcommand{\OO}{\mathcal{O}}
\renewcommand{\a}{\alpha}
\renewcommand{\b}{\beta}

\newcommand{\e}{\varepsilon}

\renewcommand{\O}{\Omega}

\newcommand{\Dom}{{\mathsf D}}

\newcommand{\g}{\gamma}
\newcommand{\embed}{\hookrightarrow}

\newcommand{\lb}{\langle}
\newcommand{\rb}{\rangle}

\newcommand{\limn}{\lim_{n\to\infty}}

\newcommand{\bX}{X}
\newcommand{\bY}{Y}

\let\mathcal \undefined
\def\mathcal{\mathscr}

\usepackage{newsymbol}
\let\emptyset \undefined
\let\ge       \undefined
\let\le       \undefined
\newsymbol\le          1336  \let\leq\le
\newsymbol\ge          133E  \let\geq\ge
\newsymbol\emptyset    203F
\newsymbol\notle       230A
\newsymbol\notge       230B

\allowdisplaybreaks

\begin{document}

\title[Stochastic reaction diffusion type equations]
{Continuous dependence on the coefficients \\ and global existence for stochastic \\ reaction diffusion equations}

\author{Markus Kunze}
\author{Jan van Neerven}

\address{Delft Institute of Applied Mathematics, Delft University of Technology, P.O. Box 5031, GA Delft, The Netherlands}
\email{M.C.Kunze@TuDelft.nl, J.M.A.M.vanNeerven@TuDelft.nl}

\date{\today}

\thanks{The authors are supported by VICI subsidy 639.033.604
of the Netherlands Organisation for Scientific Research (NWO)}

\subjclass[2000]{To be completed}


\begin{abstract} We prove convergence of the solutions $\bX_n$ of 
semilinear stochastic evolution equations on a Banach space $B$, 
driven by a cylindrical Brownian motion in a Hilbert space $H$,
\[
 \begin{aligned}
  dX_n(t) & = (A_nX(t) + F_n(t,X_n(t)))\,dt + G_n(t,X_n(t))\,dW_H(t), \\
 X_n(0) & = \xi_n,
\end{aligned}
\]
assuming that the operators $A_n$ converge to $A$ and the locally Lipschitz functions
$F_n$ and $G_n$ converge to the locally Lipschitz functions $F$ and $G$ in an appropriate sense. 
Moreover, we obtain estimates
for the lifetime of the solution $\bX$ of the limiting problem in terms of the lifetimes
of the approximating solutions $\bX_n$. 

We apply the results to prove global existence for reaction diffusion equations
with multiplicative noise and a polynomially bounded reaction term satisfying suitable dissipativity
conditions. The operator governing the linear part of the equation can be an arbitrary uniformly elliptic
second order elliptic operator.  
\end{abstract}


\maketitle

\section{Introduction}

The aim of this paper is to address the problem of continuous dependence upon the `data'
$A$, $F$, $G$, and $\xi$, of the solutions of semilinear stochastic evolution equations of the form
\begin{equation}\label{eq.scp1}\tag{SCP}\begin{aligned}
  dX(t) & = (AX(t) + F(t,X(t)))\,dt + G(t,X(t))\,dW_H(t), \\
 X(0) & = \xi,
\end{aligned}
\end{equation}
where $A$ is an unbounded linear operator on a Banach space $E$,
$W_H$ is a cylindrical Brownian motion in a Hilbert space $H$,
and $F$ and $G$ are locally Lipschitz continuous coefficients.
This continues a line of research initiated in \cite{KvN10} where the case
of globally Lipschitz continuous coefficients was considered. 
Convergence of solutions in the locally Lipschitz case considered in the present article was posed 
as an open problem in \cite{b97}.

In order to outline our approach,
we start by briefly recalling how a solution $\bX = \sol(A,F,G,\xi)$
of equation \eqref{eq.scp1}  may be found in the case of locally Lipschitz continuous coefficients
(see \cite{b97, vNVW08, s93}).

For each $r >0$ one picks functions $F^{(r)}$ and $G^{(r)}$ which are globally 
Lipschitz continuous 
and of linear growth and which coincide with $F$ and $G$ on the ball $B^{(r)} = \{x\in E: \ \n x\n\le r\}$. 
Then, denoting by $X^{(r)}$ the solution of \eqref{eq.scp1}
with $F$ and $G$ replaced with $F^{(r)}$ 
and $G^{(r)}$ respectively, one proves that with 
$$\tau^{(r)}:= \inf\{ t>0\, : \, \|X_n(t)\| > r\}$$
one has $\bX^{(r)} \equiv \bX^{(s)}$ on $[0, \tau^{(r)}]$  
for all $0<s\leq r$. In  particular, $\tau^{(r)}$ increases with $r$.
One then defines $\sigma := \lim_{r\to\infty} \tau^{(r)}$ and, for $t\in [0,\tau^{(r)}]$, 
$X(t) = X^{(r)}(t)$.  It is then shown that $\bX$ is the maximal solution of the original problem
\eqref{eq.scp1}. The stopping time $\sigma$ is called the \emph{lifetime} of $\bX$.

Suppose now that we approximate the operator $A$ by a sequence of operators $A_n$, the 
coefficients $F$ and $G$ by a sequence of coefficients $F_n$ and $G_n$, and the initial value $\xi$ by
a sequence $\xi_n$. 
For each $r>0$ 
this gives rise to processes $\bX_n^{(r)}$ from which the solution 
$\bX_n = \sol(A_n,F_n,G_n,\xi_n)$ is constructed. By the above, one expects convergence 
$\bX_n^{(r)} \to \bX^{(r)}$ as $n\to \infty$ for each $r>0$, and hence
$\bX_n^{(r)} \to \bX$ as $n\to \infty$ up to suitable stopping times. 
The aim of this paper
is to describe a general procedure which allows one to deduce that, 
in these circumstances, one indeed obtains convergence $\bX_n \to \bX$ and the lifetime of $\bX$
can be computed explicitly in terms of the stopping times 
$$\tau_n^{(r)}:= \inf\{ t>0\, : \, \|X_n(t)\| > r\}.$$
This follows from a general convergence result for processes defined up to stopping times 
presented in Section \ref{sect.1}. 

Applications to stochastic evolution 
equations are presented in Section \ref{sect.2}.
In particular, we are able to identify situations in which the limiting process $\bX$
is globally defined when the processes $\bX_n$ have this property. 

An example 
where this happens arises in the theory of stochastic reaction diffusion equations. In Section
\ref{sect.3} we prove global existence for such equations assuming that the nonlinearity
$F$  is of polynomial growth and satisfies suitable dissipativity assumptions and 
that $G$ is locally Lipschitz and of linear growth. This improves previous results 
due to Brze\'zniak and
G{\c{a}}tarek \cite{bg99} and Cerrai \cite{Cerrai} in various ways. Indeed, in our framework, the 
operator $A$ governing the linear part of the equations can be an arbitrary uniformly 
elliptic second-order operator. For such operators $A$, martingale solutions were obtained
in \cite{bg99} for polynomially bounded $F$ and uniformly bounded $G$. 
Assuming rather restrictive simultaneous diagonisability conditions on $A$ and the driving
noise, in \cite{Cerrai} global mild solutions were obtained for polynomially bounded $F$ and 
certain unbounded nonlinearities $G$. 

In Section \ref{sect.2} and \ref{sect.3} we extend these results by proving global existence
of mild solutions under the same growth assumptions on $F$ and $G$
as in \cite{Cerrai} but without any diagonisability assumptions on $A$ and the noise
process whatsoever. Although our approach combines certain essential features of 
\cite{Cerrai} with a Gronwall type lemma in the spirit of \cite{bg99}, the
the abstract results of Section \ref{sect.1} streamline the proof 
considerably. 
In the final section \ref{sect.4} we write out our results for the special case of a
1D stochastic reaction diffusion equations driven by white noise and comment
on possible variations in higher dimensions.

Notations and terminology are standard and follow those of \cite{KvN10}.
Throughout this article we fix probability space $(\Omega, \F, \P)$ endowed with a filtration
$\mathbb{F} = (\F_t)_{t\in [0,T]}$, where $0<T<\infty$ is a finite 
time horizon. Unless stated otherwise, all processes considered are defined
on this probability space, and adaptedness is understood relative to $\mathbb{F}$.
We work over the real scalar field, but occasional sectoriality arguments require
passage to complexifications; this will be done without further notice.

\section{Convergence of locally defined processes}\label{sect.1}

We begin by proving a general convergence result 
for sequences of processes defined up to certain stopping times.  For 
each $n \in \OCN := \CN \cup \{\infty\}$, a continuous adapted process
$\bX_n =(X_n(t))_{t \in [0,\sigma_n)}$ with values in a Banach space $E$ is given. 
Here, $\sigma_n: \O\to (0,T]$ denotes the explosion time of $\bX_n$, 
i.e., on the set  $\{\sigma_n < T\}$ we have
$\limsup_{t \uparrow \sigma_n}\|X_n(t)\| = \infty$. 
For each $r>0$ we set
$$\rho_n^{(r)} := \inf\big \{ t \in (0, \sigma_n) : \|X_n(t)\| > r\big\}$$
with the convention  $\inf(\emptyset)=T$.
Furthermore, we assume that for each $r>0$ we are given a globally defined, continuous, adapted process 
$\bX_n^{(r)} = (X_n^{(r)}(t))_{t \in [0,T]}$ such that the following 
conditions are satisfied:

\begin{itemize}
 \item[(a)] For all $n\in \OCN$ and $r>0$, almost surely
\[ \bX_n^{(r)}\one_{[0, \rho_n^{(r)}]} = \bX_n\one_{[0,\rho_n^{(r)}]} \ \  \mbox{on}\,\,[0,T],\]
 \item[(b)] For all $r>0$, 
\[ \limn \bX_n^{(r)} = \bX_\infty^{(r)} \ \  \mbox{in}\,\, L^0(\Omega; C([0,T]; E)).\]
\end{itemize}

Note that in (a), on the set $\{\rho_n^{(r)}=0\}$ we do require $\bX_n^{(r)}(0)= \bX_n(0)$
almost surely.
In the applications below, the processes 
$\bX_n$ are obtained by solving certain stochastic evolution equations 
with locally Lipschitz continuous
coefficients, and the processes 
$\bX_n^{(r)}$ are obtained as the solutions of the equations
with the same initial condition but with coefficients `frozen' outside the ball of radius $r$.

In the proofs of Theorem \ref{t.stopconv} and Lemma \ref{l.lem} we shall work with 
versions of $\bX_n$ and $\bX_n^{(r)}$ such that (a) holds everywhere on $\Omega$.
We denote by $B_{\rm b}([0,T];E)$ the Banach space of all bounded, strongly Borel 
measurable functions from $[0,T]$ to $E$.

\begin{thm}\label{t.stopconv}
Under the above assumptions, the following assertions hold.
\begin{enumerate}
 \item For all $r>0$ and $\eps >0$  we have, almost surely, 
\[ \liminf_{n\to\infty} \rho_n^{(r)} \leq \rho_\infty^{(r)} \leq 
\limsup_{n\to\infty} \rho_n^{(r+\eps )}\, .\]
Moreover, along every subsequence $n_k$ we can find a further subsequence $n_{k_j}$
for which we have, almost surely, 
\[ \limsup_{j\to\infty} \rho_{n_{k_j}}^{(r)} \leq \rho_\infty^{(r)} \leq 
\liminf_{j\to\infty} \rho_{n_{k_j}}^{(r+\eps )}\, .\]
 \item For all $r>0$ and $\eps >0$ we have
\[ \bX_n\one_{[0, \rho_\infty^{(r)}\wedge \rho_n^{(r+\eps)})} \to \bX_\infty\one_{[0,\rho_\infty^{(r)})}\quad
 \mbox{in}\,\, L^0(\Omega; B_{\rm b}([0,T]; E))\, .
\]
 \item We have $$\bX_n\one_{[0,\sigma_\infty \wedge \sigma_n)} \to \bX_\infty\one_{[0,\sigma_\infty )}\quad 
 \mbox{in}\,\, L^0(\Omega\times [0,T]; E).$$
\end{enumerate}
\end{thm}

\begin{rem}
 Note that the inequalities in (1) involve the whole sequences $(\rho_n^{(r)})_{n\in\N}$ and
 $(\rho_n^{(r+\e)})_{n\in\N}$. For this reason we cannot pass to an almost surely uniformly
convergent subsequence in (b) and thereby reduce the theorem to
a statement about individual trajectories (and hence to a theorem on deterministic functions). 
Limes inferior and limes superior
are highly unstable with respect to passing to a subsequence; for example, the Haar functions $h_n$
on the unit interval satisfy $\liminf_{n\to\infty} h_n = -1$ and $\limsup_{n\to\infty} h_n = 1$,
but each subsequence has a further subsequence converging to $0$ pointwise almost everywhere.  
\end{rem}

In the proof of Theorem \ref{t.stopconv} we shall use the following lemma.

\begin{lem}\label{l.lem}
For all $n \in \OCN$, $r>0$, $\eps>0$, 
and $\tau \in (0,T]$ the following holds. 
If, for some $\omega\in \Omega$,  $\|X_n^{(r+\eps)}(t,\omega)\| \leq r$ 
for all $t \in [0,\tau]$, then at least one of the following holds:
\begin{itemize}
 \item[(i)] $X_n^{(r+\eps)}(t,\omega ) =
X_n(t,\omega )$ for all $t \in [0,\tau]$;
 \item[(ii)] $\rho_n^{(s)}(\omega) = 0$ 
for all $s\in [0,r+\eps)$.
\end{itemize}
\end{lem}

\begin{proof}
We distinguish two cases. 

{\em Case 1}: If $\rho_n^{(r+\eps)}(\omega) \ge \tau$, 
then $X_n^{(r+\eps)}(t,\omega ) =
X_n(t,\omega )$ for all $t \in [0,\tau]$ by assumption (a).

{\em Case 2}: Suppose that $\rho_n^{(r+\eps)}(\omega) < \tau$ and let $s \in (r, r+\eps)$. Assume that
$\rho_n^{(s)}(\omega ) >0$. By path continuity, $0<\rho_n^{(s)}<\rho_n^{(r+\eps)}<\tau$ and 
$\|X_n(\rho_n^{(s)}(\omega), \omega )\|= s$. By (a) the contradiction
$ s = \|X_n(\rho_n^{(s)}(\omega), \omega)\| =  \|X_n^{(r+\eps)}(\rho_n^{(s)}(\omega),\omega)\| \le r$
follows. Hence we must have $\rho_n^{(s)} = 0$. 
Since $\rho_n^{(s)}=0$ for $s \in (r, r+ \eps)$, we obviously have $\rho_n^{(s)}=0$ for all 
$s \in [0,r+\eps)$.
\end{proof}

\begin{proof}[Proof of Theorem \ref{t.stopconv}]

{\em Proof of (1)} -- 
We begin with the proof of the left-hand side inequality in first assertion.

Fix $r>0$.
By (b) we may pass to a subsequence which satisfies 
$\bX_{n_k}^{(4r)} \to \bX_\infty^{(4r)}$  in $C([0,T]; E)$ almost surely,
say for all $\omega$ is a set $\Omega'$ of full probability.
Our first aim is to prove that 
\begin{align}\label{eq:claim1}\limsup_{k\to\infty}\rho_{n_k}^{(r)} 
\leq \rho_\infty^{(r)}
\end{align} 
on $\Omega'$; noting that we could also have started from an arbitrary
subsequence, this will also give the left-hand side estimate in the 
second assertion of (1). 

Fix an $\omega\in\Omega'$.
We may assume that $\rho_\infty^{(r)}(\omega)<T$, since otherwise
\eqref{eq:claim1} holds trivially. 
Likewise we may assume that $\limsup_{k\to\infty}\rho_{n_k}^{(2r)}(\omega)>0$.
For if we had $\limsup_{k\to\infty}\rho_{n_k}^{(2r)}(\omega)=0$, then
certainly $\limsup_{k\to\infty}\rho_{n_k}^{(r)}(\omega)=0$ and again \eqref{eq:claim1} 
holds trivially. 

We claim that in this situation $\rho_\infty^{(2r)}(\omega) >0$. To prove the claim, observe that since we have
$\limsup_{k\to\infty}\rho_{n_k}^{(2r)}(\omega) >0$, there is a $\delta = \delta (\omega )>0$
so that, passing to a further subsequence $\rho_{n_{k_j}}^{(2r)}  = \rho_{n_{k_{j(\omega)}}}^{(2r)} $ 
possibly depending on $\omega$, 
we have
$\rho_{n_{k_j}}^{(2r)}(\omega) \geq \delta$ for all $j$. It follows from (a) that 
$X_{n_{k_j}}(\omega ) = X_{n_{k_j}}^{(4r)}(\omega )$ on $[0,\delta]$. Moreover, $X_{n_{k_j}}^{(4r)}$ 
converges to $X_\infty^{(4r)}(\omega )$, uniformly on $[0,\delta]$. Hence also $X_{n_{k_j}}(\omega)$
converges to $X_\infty^{(4r)}(\omega )$, uniformly on $[0,\delta]$. Now, since 
$\|X_{n_{k_j}}(t,\omega )\| \leq 2r$ for $t \in [0,\delta]$, it follows that $\|X_{n_{k_j}}^{(4r)}(t,\omega )\|\leq 2r$
for $t\in[0,\delta]$ which, by Lemma \ref{l.lem}, implies that $X_{n_{k_j}}(t,\omega ) = X^{(4r)}_{n_{k_j}}(t,\omega )$ for such 
$t$. By passing to the limit $j\to\infty$ we find $\|X_\infty (t,\omega )\|\leq 2r$ for $t \in [0,\delta]$ and thus $\rho_\infty^{(2r)}(\omega ) \geq \delta >0$.
This proves the claim.

We can now choose
a sequence $t_j (\omega) \downarrow \rho_\infty^{(r)}(\omega )$ such 
that $t_1(\omega) < \rho_\infty^{(2r)}(\omega)$ 
and $\|X_\infty (t_j(\omega ), \omega )\| > r$ for all $j$. 
Such a sequence exists by 
our assumption 
that $\rho_\infty^{(r)}(\omega )<T$, 
the definition of $\rho_\infty^{(r)}(\omega)$,   
and path continuity. 
For each $j$ there is an index $k_0(\omega,j)$ such that 
\[ \|\bX_{n_k}^{(4r)}(\omega ) -  \bX_\infty^{(4r)}(\omega )\|_{C([0,T]; E)} < 
\min\big\{ \|X_\infty(t_j(\omega ), \omega)\| -r, r\big\}\]  for 
all $k \geq k_0(\omega,j)$. 
For such $k$ we have $$\|X_{n_k}^{(4r)}(t,\omega)\| < 3r \ \hbox{ for all } 
0\leq t \leq \rho_\infty^{(2r)}(\omega).$$
To see this, note that if $0\leq t \leq \rho_\infty^{(2r)}(\omega)$,
then $\n X^{(4r)}_\infty (t, \omega)\n =  \n X_\infty (t, \omega)\n \le 2r$. 
Also, for all such $k$ we have
$$\|\bX_{n_k}^{(4r)}(t_j(\omega), \omega )\| > r.$$ 
By Lemma \ref{l.lem}, either $\|\bX_{n_k}(t_j(\omega),\omega)\| 
>r$ or $\rho_{n_k}^{(r)}(\omega) = 0$. 
Note that in both cases, 
$$\rho_{n_k}^{(r)}(\omega ) \leq t_j(\omega).$$ 
This being true for all $k\ge k_0(\omega,j)$, it follows that
$\limsup_{k\to \infty} \rho_{n_k}^{(r)}(\omega ) \leq t_j(\omega )$. 
Taking the infimum over $j$, we see that 
$\limsup_{k\to\infty}\rho_{n_k}^{(r)}(\omega) 
\leq \rho_\infty^{(r)}(\omega)$. This proves \eqref{eq:claim1}.

Now fix $\eta >0$. On the set 
$\bigcup_{m\in\CN}\bigcap_{n\geq m}\{ \rho_n^{(r)} 
\geq \rho_\infty^{(r)} + \eta\}$,
the above subsequence certainly satisfies 
$\limsup_{k\to\infty} \rho_{n_k}^{(r)} \geq \rho_\infty^{(r)} + \eta$.
But since \eqref{eq:claim1} holds on a set of full probability, this implies that
$\P (\bigcup_{m\in\CN}\bigcap_{n\geq m} 
\{ \rho_n^{(r)} \geq \rho_\infty^{(r)} + \eta\})= 0$. It follows that
\[ \P\big( \liminf_{n\to\infty} \rho_n^{(r)} \leq \rho_\infty^{(r)} + \eta\big) 
 \geq \P \big(\bigcap_{m\in\CN}\bigcup_{n\geq m} \{ \rho_n^{(r)} \leq 
\rho_\infty^{(r)} + \eta \}\big) = 1 \, .
\]
Upon letting $\eta \downarrow 0$ 
we have $\{ \liminf_{n\to\infty}\rho_n^{(r)} \leq \rho_\infty^{(r)}+ \eta\} 
\downarrow \{ \liminf_{n\to\infty} \rho_n^{(r)}
\leq \rho_\infty^{(r)}\}$, from which it 
follows that $\P(\liminf_{n\to \infty}\rho_n^{(r)} 
\leq \rho_\infty^{(r)})=1$.
\medskip 

Next we prove the right-hand side inequality of the first assertion in (1).

Fix $r>0$ and $\e>0$.
By (b) we may pass to a subsequence such that
$\bX_{n_k}^{(r+2\e)} \to \bX_\infty^{(r+2\e)}$ in $C([0,T]; E)$ almost surely,
say on the set $\Omega'$ of full probability. 
Our first aim is to prove that 
\begin{align}\label{eq:claim2}
\liminf_{k\to\infty}\rho_{n_k}^{(r+\eps)}\geq \rho_\infty^{(r)}
\end{align}
on $\Omega'$; noting that we could also have started from an arbitrary
subsequence, this will also give the right-hand side estimate in the 
second assertion of (1). 

Fix an $\omega\in\Omega'$.
We may assume that 
$\rho^{(r)}_\infty(\omega)>0$, for otherwise \eqref{eq:claim2}
trivially holds. 

The next step is to prove that $X_{n_k}(\omega)\to X_\infty(\omega)$ 
uniformly on $[0,\rho^{(r)}_\infty(\omega)]$.
On this interval we know that 
$\n X_{\infty}(\omega)\n \le r$. Hence, by (a), 
$X_{\infty}(\omega) = X_{\infty}^{(r+2\e)}(\omega)$
on $[0,\rho^{(r)}_\infty(\omega)]$.
Since $X_{n_k}^{(r+2\e)}(\omega) \to X_{\infty}^{(r+2\e)}(\omega)$ uniformly,
it follows that, for large enough $k$, say for all $k\ge k_1(\omega)$, 
\begin{align}\label{eq:extra}
\n X_{n_k}^{(r+2\e)}(\omega)\n \le r+\e \ \hbox{ on } \ [0,\rho^{(r)}_\infty(\omega)].
\end{align}
By \eqref{eq:extra} and Lemma \ref{l.lem}, for each $k\ge k_1(\omega)$ 
we are in at least one of the following two cases: 
either we have $\|X_{n_k}(\omega)\| \le r + \eps$ 
on $[0,\rho_\infty^{(r)}(\omega)]$ and 
thus $\rho_{n_k}^{(r+\eps)}(\omega) \geq \rho_\infty^{(r)}(\omega)$, or else 
$\rho_{n_k}^{(r+\eps)}(\omega) = 0$. 

Suppose the latter happens for infinitely many $k$  (the set of these $k$ may depend on $\omega$). Then 
$\|X_{n_k}(0,\omega)\| \geq r+\eps$ for infinitely many $k$.
Since
$$X_{n_k}(0,\omega) = X_{n_k}^{(r+2\e)}(0,\omega) \to X_\infty^{(r+2\e)}(0,\omega) 
=X_\infty (0,\omega)$$ 
this implies $\|X_\infty(0,\omega)\|
\geq r+\eps$. But then $\rho_\infty^{(r)}(\omega) = 0$ by path continuity, and this 
contradicts our previous assumption. 
Thus, for all but finitely many $k$ we must have the first alternative. 
This proves \eqref{eq:claim2}.

Fix $\eta >0$. Arguing as above, $\P (\bigcup_{m\in\CN}\bigcap_{n\geq m} 
\{ \rho_n^{(r+\eps)} \leq 
\rho_\infty^{(r)} - \eta \}) = 0$ and thus 
$\P (\limsup_{n\to \infty} \rho_n^{(r+\eps)} \geq \rho_\infty^{(r)} - \eta ) = 1$.
Upon letting $\eta \downarrow 0$ we see that $\P (\limsup_{n\to \infty} 
\rho_n^{(r+\eps)} \geq \rho_\infty^{(r)} ) = 1$.
\medskip

{\em Proof of (2)} --  Fix $r>0$ and $\eps>0$. Since convergence in probability is metrisable, 
it suffices to prove that 
every subsequence of $(\bX_n\one_{[0, \rho_\infty^{(r)}\wedge 
\rho_n^{(r+\eps)})})_{n\in\CN}$ has a further 
subsequence for which the claimed convergence holds. 

Given a subsequence, we may pass to a further subsequence
(which, for ease of notation, we index by $n$ again) such that 
 \begin{align}\label{eq:as}
  \bX_n^{(r+2\eps)} \to \bX_\infty^{(r+2\eps)} \ \ \hbox{in} \  C([0,T]; E) \ \hbox{ 
 almost surely. }
\end{align}
Fix an $\omega$ from the set of convergence. If $\rho_\infty^{(r)}(\omega)=0$,
 then it is trivial that 
 $ \bX_n\one_{[0, \rho_\infty^{(r)}\wedge \rho_n^{(r+\eps)})}(\omega) 
 \to \bX_\infty\one_{[0,\rho_\infty^{(r)})}(\omega)$, and therefore we may
 assume that $\rho_\infty^{(r)}(\omega)>0$. 
Then, as we have seen in the proof of the second assertion of (1), 
for all $n\geq n_0(\omega)$ we have 
$\|X_n(t,\omega)\|\leq r+\eps$ for all $0\leq t \leq \rho_\infty^{(r)}(\omega)$.
For these $n$ we see that $ \rho_n^{(r+\eps)}(\omega) \ge \rho_\infty^{(r)}(\omega) $
and therefore $\bX_n\one_{[0,\rho_\infty^{(r)}\wedge \rho_n^{(r+\eps)})}(\omega) 
= \bX_n\one_{[0,\rho_\infty^{(r)})}(\omega)$. 
Also, $X_n(t,\omega ) = X_n^{(r+2\eps)}(t,\omega)$ and $X_\infty (t,\omega) = 
X_\infty^{(r+2\eps)}(t,\omega)$ for $0\leq t\leq \rho_\infty^{(r)}(\omega) \wedge \rho_n^{(r+\eps)}(\omega)$.
Combining these observations with \eqref{eq:as} 
we find, for $n\ge n_0(\omega)$, 
\begin{align*}\bX_n\one_{[0,\rho_\infty^{(r)}\wedge 
\rho_n^{(r+\eps)})}(\omega) 
  = \bX_n\one_{[0,\rho_\infty^{(r)})}(\omega) 
& = \bX_n^{(r+2\e)}\one_{[0,\rho_\infty^{(r)})}(\omega)
\\ & \to \bX_\infty^{(r+2\e)} \one_{[0,\rho^{(r)}_\infty )}(\omega)
  = \bX_\infty \one_{[0,\rho^{(r)}_\infty )}(\omega)
\end{align*} 
in $B_{\rm b}([0,T]; E)$. 
\medskip 

{\em Proof of (3)} -- Again, we will show that every subsequence has a subsequence for which
the claimed convergence holds. 

Let a subsequence be given. 
By the proof of (2), this subsequence has a further subsequence $n_{k,1}$ such that
\[ \bX_{n_{k,1}}\one_{[0,\rho_\infty^{(1)}\wedge 
\rho_{n_{k,1}}^{(2)})}(\omega) \to \bX_\infty \one_{[0,\rho^{(1)}_\infty )}(\omega)\]
in $B_{\rm b}([0,T]; E)$ as $k\to \infty$, for all $\omega$ outside a set null set $N_1$.

Suppose we have already constructed a subsequence $n_{k,l}$ such 
that
\[ \bX_{n_{k,l}}\one_{[0,\rho_\infty^{(j)}\wedge 
\rho_{n_{k,l}}^{(j+1)})}(\omega) \to \bX_\infty 
\one_{[0,\rho^{(j)}_\infty )}(\omega)\]
in $B_{\rm b}([0,T]; E)$ as $k \to \infty$, 
 for all $j \in \{1,\ldots,l\}$ and all $\omega$ outside a null set $N_l$.
By the proof of (2), we can extract a further subsequence $n_{k,l+1}$ 
such that 
\[ \bX_{n_{k,l}}\one_{[0,\rho_\infty^{(j)}\wedge 
\rho_{n_{k,l+1}}^{(j+1)})}(\omega) \to \bX_\infty 
\one_{[0,\rho^{(j)}_\infty )}(\omega)\]
in $B_{\rm b}([0,T]; E)$ as $k \to \infty$,
for all $j \in \{1, \ldots,l,l+1\}$ and all $\omega$ outside a 
null set $N_{l+1}$. We continue this procedure inductively.

Now put $N := \bigcup_{l\geq 1} N_l$. 
Setting $n_l := n_{l,l}$, it follows that 
\begin{equation}\label{eq.diag}
 \bX_{n_l}\one_{[0,\rho_\infty^{(j)}\wedge 
\rho_{n_l}^{(j+1)})}(\omega) \to \bX_\infty \one_{[0,\rho^{(j)}_\infty )}(\omega)
\end{equation}
in  $B_{\rm b}([0,T]; E)$ as $l\to \infty$, for all $j\ge 1$
and $\omega$ outside the null set $N$.

By the second part of (1),
upon replacing $N$ by some larger null set and passing to a further subsequence of $n_l$ if necessary,
we may assume that outside $N$ we also have 
\begin{align}\label{eq:claim4}
\liminf_{l\to\infty}\rho_{n_l}^{(j+1)}(\omega)\ge  \rho_\infty^{(j)} (\omega) \ \hbox{ for all } \ j \ge 1.
\end{align} 
Now let $(t,\omega )\in [0,T]\times (\Omega\setminus N)$.
We claim that 
\[ X_{n_l}(t,\omega )\one_{[0,\sigma_\infty \wedge \sigma_{n_l})}(t,\omega) 
\to X_\infty (t,\omega)\one_{[0,\sigma_\infty )}(t,\omega) \]
in $E$ as $l \to \infty$. 

We distinguish two cases.
First, if $t\geq \sigma_\infty (\omega)$, then  
$$X_{n_l}(t,\omega)\one_{[0,\sigma_\infty \wedge \sigma_{n_l})}(t,\omega)
= 0 =  X_\infty (t,\omega)\one_{[0,\sigma_\infty )}(t,\omega)$$
for all $l\in\CN$ and 
there is nothing to prove. 

Second, suppose that $t < \sigma_\infty (\omega)$.
Pick an integer $j$ such that 
$\|X_\infty (s,\omega )\| < j$
for all $0\leq s \leq t$. Then $t< \rho_\infty^{(j)}(\omega)$. 
By \eqref{eq:claim4}, 
for all large enough $l$ we have $t< \rho_{n_l}^{(j+1)}(\omega)\le \sigma_{n_l}(\omega).$ 
Hence, for all large $l$, 
$$
  X_{n_l}(t,\omega)\one_{[0, \sigma_\infty\wedge \sigma_{n_l})}(t,\omega) 
= X_{n_l}(t,\omega) = X_{n_l}(t,\omega)\one_{[0,\rho_\infty^{(j)}
\wedge \rho_{n_l}^{(j+1)})}(t,\omega).
$$ 
By \eqref{eq.diag}, the right-hand side converges to 
$$X_\infty(t,\omega) = X_\infty(t,\omega)\one_{[
0,\rho_\infty^{(j)})}(t,\omega) = X_\infty(t,\omega)\one_{[
0,\sigma_\infty)}(t,\omega).$$ This proves the claim.
\end{proof}

\begin{cor}
Under the above assumptions we have $$\sigma_\infty \ge
\lim_{r\to\infty}\liminf_{n\to\infty}\rho_n^{(r)}$$ almost surely.
Furthermore, every subsequence $n_k$ has a further subsequence $n_{k_j}$ for which
 $$\sigma_\infty =
\lim_{r\to\infty}\liminf_{j\to\infty}\rho_{n_{k_j}}^{(r)}$$ almost surely.
\end{cor}

\begin{proof}
The first assertion follows from 
the first assertion in Theorem \ref{t.stopconv}(1) upon letting $r\to\infty$.
To obtain the second assertion,
given a subsequence $n_k$ let $n_{k_j}$ be a subsequence for which the second assertion in Theorem \ref{t.stopconv}(1) holds.
Then  $\sigma_\infty \le
\lim_{r\to\infty}\liminf_{j\to\infty}\rho_{n_{k_j}}^{(r)}$ almost surely. The reverse
inequality follows from the first part of Theorem \ref{t.stopconv}(1) applied to this subsequence. 
\end{proof}

\begin{cor}\label{c.hoelderconv}
Under the above assumptions, 
suppose that $\sigma_n = T$ almost surely for all $n \in \CN$, and suppose furthermore
that for some $p\geq 1$
we have $$\sup_{n\in\CN} \|\bX_n \|_{L^p(\Omega; C([0,T]; E ))} < \infty.$$
Then:
\begin{enumerate}
 \item Almost surely, $\sigma_\infty = T$;
\item We have $X_\infty\in L^p(\Omega;C([0,T];E))$;
 \item If $p>1$, then, for all $1\leq q<p$,  
$$\bX_n \to \bX_\infty \ \hbox{ in } \ L^q(\Omega; C([0,T]; E)).$$ 
\end{enumerate}
\end{cor}

\begin{proof}
(1) From Theorem \ref{t.stopconv}(2) and Fatou's lemma we infer, for $r>0$ and $\eps >0$,
\[ 
\begin{aligned}
\expect\|\bX_\infty\one_{[0,\rho_\infty^{(r)})}\|_{B_{\rm b}([0,T]; E)}^p & \leq \liminf_{n\to \infty}
 \expect\| \bX_n\one_{[0,\rho_\infty^{(r)}\wedge\rho_n^{(r+\eps)})}\|_{B_{\rm b}([0,T];E)}^p\\
&\leq \liminf_{n\to\infty}\expect\|\bX_n\|_{C([0,T]; E)}^p \leq C  ,
\end{aligned}
\]
where $C:= \sup_{n\in\CN} \|\bX_n \|_{L^p(\Omega; C([0,T]; E))}$. 
Employing Fatou's lemma a second time, we see that 
\begin{align*}
\E \n X_\infty \one_{[0,\sigma_\infty)}\n_{B_{\rm b}([0,T];E)}^p
& =  \expect \lim_{r \to \infty} \|\bX_\infty\one_{[0,\rho_\infty^{(r)})}\|_{B_{\rm b}([0,T]; E)}^p 
\\ & \leq \liminf_{r\to\infty}\expect\|\bX_\infty\one_{[0,\rho_\infty^{(r)})}\|_{B_{\rm b}([0,T]; E)}^p \leq C\, .
\end{align*}
In particular, we infer that $\bX_\infty$ is almost surely bounded on $[0,\sigma_\infty )$. Since 
$\sigma_\infty$ is an explosion time, this is only 
possible if $\sigma_\infty = T$.\medskip  

(2) From what we have proved so far it follows that $X_\infty\in L^p(\Omega;C_{\rm b}([0,T);E))$ and 
thus $\sup_{t\in[0,T)}\|X_\infty (t,\omega )\|< \infty$ for almost all $\omega \in \Omega$.
This implies that $\bigcup_{k\in\CN}\big\{\rho_\infty^{(k)}=T\big\}$ has full measure and hence, by (a), 
$X(\omega ) \in C([0,T]; E)$ almost surely. By continuity, $\sup_{t\in [0,T)}\|X_\infty (t,\omega )\|
= \sup_{t\in [0,T]}\|X_\infty (t,\omega )\|$ almost surely and now $X_ \infty 
\in L^p(\Omega; C([0,T]; E))$ follows.
\medskip 
 
(3) Fix $1\leq q<p$ and write, for $r>0$ and $\e>0$, 
\begin{equation}\label{eq.convest}
\begin{aligned}   \expect \|\bX_n - \bX_\infty\|_{C_{\rm b}([0,T]; E)}^q 
& \lesssim 
\expect\big\|\bX_n \big(\one_{[0,T)}-\one_{[0,\rho_\infty^{(r)}\wedge\rho_n^{(r+\eps)})}\big)
\big\|_{B_{\rm b}([0,T]; E)}^q\\ 
& \quad+ \expect\big\|\bX_n\one_{[0,\rho_\infty^{(r)}\wedge\rho_n^{(r+\eps)})} - 
\bX_\infty\one_{[0,\rho_\infty^{(r)})}\big\|_{B_{\rm b}([0,T];E)}^q\\
&\quad + \expect \big\|\bX_\infty \big(\one_{[0,\rho_\infty^{(r)})}-\one_{[0,T)}\big)\big\|_{B_{\rm b}([0,T];E)}^q \, .
\end{aligned}
\end{equation}
From
\[ \big\{ \rho_\infty^{(r)}\wedge \rho_n^{(r+\eps)} < T\big\} \subseteq
 \big\{\|\bX_\infty\|_{B_{\rm b}([0,T]; E)} \geq r\big\} \cup 
 \big\{\|\bX_n\|_{B_{\rm b}([0,T]; E)} \geq r +\eps\big\}\, ,
\]
it follows that
\begin{equation}\label{eq:twoterms}
\begin{aligned}\ & \big\|\bX_n\big(\one_{[0,T)}-\one_{[0,\rho_\infty^{(r)}\wedge
\rho_n^{(r+\eps))})}\big)\big\|_{B_{\rm b}([0,T];E)} 
\\ & \ \ \leq
\one_{\{\|\bX_\infty\|_{B_{\rm b}([0,T];E)} \geq r\}}\|\bX_n\|_{B_{\rm b}([0,T]; E)}\vee 
\one_{\{\|\bX_n\|_{B_{\rm b}([0,T];E)} \geq r +\eps\}}\|\bX_n\|_{B_{\rm b}([0,T]; E)}.                  
\end{aligned}
\end{equation}
Since the $\bX_n$ form a bounded sequence in $L^p(\Omega; C([0,T]; E))$, 
the random variables $\|\bX_n\|_\infty^q$ are uniformly integrable.
Using this to handle the second term on the right-hand side of \eqref{eq:twoterms}
and H\"older's inequality (with exponents $p$ and $p'$, $\frac1p+\frac1{p'}=1$) 
to handle the first,
\eqref{eq:twoterms}
implies that the first term on the right-hand side of \eqref{eq.convest} 
converges to $0$ as $r\to\infty$. 

For fixed $r>0$ and $\eps >0$, convergence of the middle term on the right-hand side of \eqref{eq.convest} 
follows from Theorem \ref{t.stopconv} (2) and the uniform integrability of the $\|\bX_n\|_\infty^q$. 

As for the third term, arguing as before we have
$$\|\bX_\infty \big(\one_{[0,T)}-\one_{[0,\rho_\infty^{(r)})}\big)\big\|_{B_{\rm b}([0,T];E)}
\le \one_{\{\|\bX_\infty\|_{B_{\rm b}([0,T];E)} \geq r\}}\|\bX_\infty\|_{B_{\rm b}([0,T]; E)}
$$
and therefore the third term on the right-hand side of \eqref{eq.convest} tends to $0$ as $r\to \infty$.

It thus follows that $X_n\to X_\infty$ in $L^q(\Omega;C_{\rm b}([0,T];E))$. 
\end{proof}

\section{Application to semilinear stochastic equations}\label{sect.2}

We shall now apply the abstract results of the previous section to prove 
convergence of approximate solutions of stochastic evolution equations
of the form
\begin{equation}\label{eq.scp}\tag{SCP}
\left\{ \begin{aligned}
         dX(t) & =  [ AX(t) + F(t,X(t))]\,dt + G(t,X(t))\,dW(t)\\
	  X(0) & =  \xi.
        \end{aligned}
 \right.
\end{equation}
The driving noise process $W$ is assumed to be a cylindrical 
Brownian motion in some Hilbert space $H$.

\subsection{Stochastic evolution equations in UMD spaces}\label{sec:UMD}

Under the assumptions stated below, existence and uniqueness of maximal 
solutions for \eqref{eq.scp} in UMD spaces $E$
 was proved in \cite{vNVW08}, and convergence of the solutions in the case of globally Lipschitz 
continuous coefficients was established in \cite{KvN10}.

Continuing the notations of the previous section we shall write $A=A_\infty$, $F=F_\infty$, 
$G=G_\infty$ and $\xi = \xi_\infty$ when we thinks of these objects as the limits of sequences
of approximating objects $A_n$, $F_n$, $G_n$, $\xi_n$. 

\begin{itemize}
 \item[(A1)] For $n \in \OCN$, the operators $A_n$ are densely defined, closed and \emph{uniformly sectorial} 
on $E$ in the
 sense that there exist numbers $M\geq 1$ and $w \in \CR$ such that each $A_n$ is sectorial of type 
$(M,w)$.
 \item[(A2)] The operators $A_n$ converge to $A_\infty$ \emph{in the strong resolvent sense}:
\[ \lim_{n\to\infty} R(\lambda, A_n)x = R(\lambda, A_\infty)x \]
for some (equivalently, all) $\Re\lambda > w$ and all $x \in E$. 
\end{itemize}
Assumptions (A1) and (A2) coincide with those made in 
\cite{KvN10}. 
Assuming (A1), the strongly continuous analytic semigroups $(e^{-wt}S(t))_{t\ge 0}$
are uniformly bounded, uniformly in $n$.
Therefore, for $w'>w$ the fractional powers $(w^\prime -A_n )^\a$ 
are well defined for all $\a \in (0,1)$.
In particular, the fractional domain spaces 
$$E_{n,\a} := \Dom((w^\prime-A_n)^\a)$$ are Banach spaces with respect to the norm 
$$\|x\|_{E_{n,\a}} := \|(w^\prime -A_n )^\a x\|.$$ Up to equivalent 
norms, these spaces are independent of the choice of $w^\prime$.
It may happen, however, that these spaces vary with $n$. This may cause problems,
and to avoid these we make the  following assumption.

\begin{itemize}
 \item[(A3)] For all $0<\a<\half$ we have $E_{n,\a} = E_{\infty, \a}$ as 
linear subspace of $E$. 
Moreover, there exist constants $c_\a>0$ and $C_\a>0$ such that
\[ c_\a\|x\|_{E_{\infty,\a}} \leq \|x\|_{E_{n,\a}} \leq C_\a\|x\|_{E_{\infty,\a}}\quad \forall \, x \in E_\a\, , \
n \in \CN\, . \]
\end{itemize}
We then set $E_\a := E_{\infty, \a}$ and $\|\cdot\|_\a 
:= \|\cdot\|_{E_{\infty, \a}}$.
We complete the scale $E_\a$ by setting $E_0 := E$ and $\|\cdot\|_0 := \|\cdot\|$.

It is immediate from assumption (A3) that 
for each $0<\a<\half$, 
the operators $(w^\prime - A_n)^\a$
are uniformly bounded in $\cL (E_\a, E)$ 
and that the operators  $(w^\prime - A_n)^{-\a}$ are uniformly bounded 
in $ \cL (E, E_\a )$.

For $0<\a< \half$ we define the extrapolation spaces $E_{n, -\a}$ 
as the completion of $E$ under the norms 
$\|x\|_{E_{n,-\a}} := \|(w^\prime - A_n)^{-\a}x\|_E$.
For fixed $n$, these spaces are independent of $w'>w$ up to an equivalent norm,
and for each fixed $w'>w$ these spaces are independent of $n$ 
with equivalence constants independent of $n$. 
Accordingly, we set $E_{-\a} := E_{\infty, - \a}$ and $\|\cdot\|_{-\a} 
:= \|\cdot\|_{E_{\infty, -\a}}$.
Then for all $0 \leq \a,\b < \half$, 
the operators
$(w^\prime - A_n)^{\a+\b}$ and $(w^\prime - A_n)^{-(\a+ \b )}$ are uniformly bounded 
 in $\cL (E_\a, E_{-\b})$ and $\cL (E_{-\b}, E_\a )$, respectively.
\medskip 

Concerning the coefficients $F_n$ and $G_n$, we shall assume that the hypotheses
of \cite[Section 8]{vNVW08} are satisfied, uniformly with respect to $n$, and with 
exponents 
$$0\le \theta<\tfrac12, \quad 0\le \kappa_F,\kappa_G< \tfrac12,$$
and we add the assumptions concerning their convergence of \cite{KvN10}. 
The precise assumptions are as follows. Recall that $\g(H,F)$ denotes the operator ideal
of all $\g$-radonifying operators from $H$ into the Banach space $F$ (see \cite{vNVW08}
for further explanations of the hypotheses involving these spaces).

\begin{itemize}
 \item[(F1)] The maps $F_n : [0,T]\times \Omega \times E_\theta \to E_{-\kappa_F}$ 
are 
uniformly locally Lipschitz continuous, i.e.,  for all $r>0$ there exists a 
constant $L_F^{(r)}\geq 0$ such that 
\[ \|F_n(t,\omega, x) - F_n(t,\omega, y)\|_{-\kappa_F} \leq L_F^{(r)}\|x-y\|_\theta \]
for all $t\in [0,T]$, $\omega \in \Omega$ and $x,y \in E_\theta$ of norm $\|x\|_\theta, \|y\|_\theta\leq r$.
Moreover,
for all $x \in E_\theta$ the map $(t,\omega )\mapsto F_n(t,\omega, x)$ 
is strongly measurable and adapted and there exists a constant $C_{F,0}$ such that  
$$ \| F(t,\omega,0)\|_{E_{-\kappa_F}} \le C_{F,0}.$$

\item[(F2)] For all $r>0$ and almost all $(t,\omega ) \in [0,T]\times \Omega$ 
we have 
$$F_n^{(r)}(t, \omega, x) \to F_\infty^{(r)} (t, \omega, x) \ \hbox{ in } \, 
E_{-\kappa_F}$$ for  all $x \in E_\theta$.

 \item[(G1)] The maps $G_n : [0,T]\times\Omega \times E_{\theta} \to \gamma (H, E_{-\kappa_G})$ 
are uniformly locally $\gamma$-Lipschitz continuous, i.e., for all $r>0$ there exist maps $G_n^{(r)} :
[0,T]\times \Omega\times E_\theta \to \gamma (H, E_{-\kappa_G})$ such that
\[ G_n^{(r)} = G_n \quad\mbox{on} \quad [0,T]\times\Omega\times\{x \in E_\theta\, : \, \|x\|_\theta \leq r\}. \] 
Moreover, there there exist constants $L_G^{(r)}$ 
such that for all Borel probability measures $\mu$ on $[0,T]$, all $\omega \in \Omega$, 
all  $\phi_1, \phi_2 \in L^2([0,T],\mu; E_\theta)\cap \gamma (L^2([0,T],\mu), E_\theta)
=: L^2_\gamma ([0,T], \mu; E_\theta )$, and all $n \in \OCN$ we have
\begin{align*}
\ & \qquad \|G_n^{(r)} (\cdot , \omega, \phi_1 ) - G_n^{(r)} (\cdot , \omega, \phi_2)\|_{\gamma (L^2([0,T], \mu; H), E_{-\kappa_G})}
\\ & \hskip6cm \leq  L_G^{(r)}\|\phi_1 -\phi_2 \|_{L^2_\gamma ([0,T], \mu; E_\theta)}.
\end{align*}
For all $x \in E_\theta$, $h \in H$, and $n \in \OCN$ 
there exists a constant $C_{G,0}$ such that  for all Borel probability measures $\mu$ on $[0,T]$,
\begin{align*}
 \|G_n^{(r)} (\cdot, \omega, 0)\|_{\gamma (L^2([0,T], \mu; H), E_{-\kappa_G})}
 \leq & C_{G,0}.
\end{align*}
Finally, we assume that for all $n \in \OCN\,,\, x \in E_\theta$ and $h \in H$ the map 
$(t,\omega ) \mapsto G_n(t,\omega, x)h$ is strongly measurable and adapted. We also assume this 
measurability and adaptedness of the maps $G_n^{(r)}$.

\item[(G2)] For all $r>0$ and almost all $(t,\omega ) \in [0,T]\times \Omega$ we have 
$$G_n^{(r)}(\cdot, \omega, x) \to G_\infty^{(r)} (\cdot, \omega, x)
\ \hbox{ in }\, \gamma (L^2(0,T,\mu; H), E_{-\kappa_G})$$ for 
all $x \in E_\theta$ and all Borel probability measures $\mu$ on $[0,T]$.
\end{itemize}
Examples where these assumptions are satisfied have been presented in \cite{KvN10,vNVW08}.
When $E$ also has type $2$, then the conditions (G1) and (G2) are implied by the `classical'
notions of Lipschitz continuity and convergence assumptions, respectively, 
with respect to the norm of $\g(H,E_{-\kappa_G})$; see \cite[Lemma 5.2]{vNVW08} (cf. the statement
of Proposition \ref{p.banachlip}).

For UMD spaces $E$, under the above assumptions the existence of a 
unique maximal 
solution $(X_n(t))_{t \in [0,\sigma_n)}$ of \eqref{eq.scp} 
with coefficients $A_n$, $F_n$, $G_n$ was proved in \cite[Theorem 8.1]{vNVW08} 
for initial data $\xi_n\in L^p(\Omega,\F_0,\P;E_\theta)$ with $2<p<\infty$. Moreover, it was shown that 
$\sigma_n$  is an explosion time for $X_n$.
 In this context we shall write
$$ X_n = \sol (A_n, F_n, G_n, \xi_n).$$

In the special case when the coefficients $F_n$ and $G_n$ are of linear growth and 
satisfy global Lipschitz 
assumptions (so that $\sigma_n \equiv T$), the convergence results proved in
 \cite[Theorems 4.3, 4.7]{KvN10}  for
the case $\theta = \kappa_F = \kappa_G = \delta = 0$ 
can be extended {\em mutatis mutandis} to yield the following result.

\begin{prop}\label{prop:global}
Let $E$ be a UMD space, assume (A1), (A2), (A3), 
suppose the mappings
$F_n : [0,T]\times\Omega\times E_{\theta} \to E_{-\kappa_F}$ and
$G_n : [0,T]\times\Omega\times E_{\theta} \to \gamma (H, E_{-\kappa_G})$
satisfy the global Lipschitz counterparts 
of (F1), (G1) with linear growth assumptions, and assume that they satisfy (F2), (G2). 
Let $2<p<\infty$, $0\le \theta<\frac12$, $0\le \kappa_F, \kappa_G<\frac12$ 
satisfy 
\begin{align}\label{eq:exps}
\theta + \kappa_F < \tfrac{3}{2} - \tfrac{1}{\tau}, \quad 
\theta + \kappa_G < 1 - \tfrac{1}{p} - \tfrac{1}{\tau},
\end{align} 
where 
$\tau\in (1,2]$ denotes the type of $E$. 
If $\xi_n \to \xi_\infty$ in $L^p(\Omega, \cF_0, \P;E_\theta)$,
then the global solutions $(X_n)_{t\in [0,T]}$ of  \eqref{eq.scp} satisfy
\begin{align*}
\bX_n \to \bX_\infty \ \hbox{ in } \, L^q(\Omega; C([0,T]; E_\theta )).
\end{align*}
If $\lambda, \delta \geq 0$ satisfy 
\begin{align}\label{eq:exps2}
\lambda + \delta < \tfrac12 -\tfrac1p -\kappa_G,
\end{align}
then for all  $1\leq q < p$,
\begin{align*}
 \bX_n - S_n(\cdot )\xi_n \to \bX_\infty - S_\infty(\cdot )\xi_\infty \ \hbox{ in } \, 
L^q( \Omega; C^\lambda ([0,T]; E_\delta )).
\end{align*} 
\end{prop}

The approximation of $A$ by the operators $A_n$ 
with respect to the fractional domain norms is handled by using Lemma \ref{l.fracconv};
for the rest the proofs of \cite[Theorems 4.3, 4.7]{KvN10} carry over almost word for word. 

\begin{rem}\label{rem:A3}
In situations where one has $\theta = \kappa_F = \kappa_G = \delta = 0$
with $F$ and $G$ not necessarily globally Lipschitz continuous, Assumption (A3) is not needed 
in Proposition \ref{prop:global} and also not in the following results.
\end{rem}

Combining this result with Theorem \ref{t.stopconv}, we obtain the following 
extension of Proposition \ref{prop:global}
to the locally Lipschitz case.

\begin{thm}\label{t.main} Let $E$ be a UMD space, assume (A1), (A2), (A3), (F1), (F2), (G1), (G2), 
and let \eqref{eq:exps} hold. Suppose that 
$\xi_n\to \xi_\infty$ in $L^p(\Omega,\F_0,\P;E_\theta)$.
Let $(X_n(t))_{t\in [0,\sigma_n)} = \sol (A_n, F_n, G_n, \xi_n)$
and define $$\rho_n^{(r)} := \inf\big\{ t \in (0, \sigma_n)\, : \|X_n(t)\|_\theta > r\big\}.$$ 
Then, 
\begin{enumerate}
\item For all $r>0$ and $\eps >0$  we have, almost surely, 
\[ \liminf_{n\to\infty} \rho_n^{(r)} \leq \rho_\infty^{(r)} \leq 
\limsup_{n\to\infty} \rho_n^{(r+\eps )}\, ;\]
 \item For all $r>0$ and $\eps >0$ we have
\[ \bX_n\one_{[0, \rho_\infty^{(r)}\wedge \rho_n^{(r+\eps)})} \to \bX_\infty\one_{[0,\rho_\infty^{(r)})}\quad
 \mbox{in}\,\, L^0(\Omega; B_{\rm b}([0,T]; E_\theta))\, ;
\]
 \item We have $$\bX_n\one_{[0,\sigma_\infty \wedge \sigma_n)} \to \bX_\infty\one_{[0,\sigma_\infty )}\quad 
 \mbox{in}\,\, L^0(\Omega\times [0,T]; E_\theta).$$
\end{enumerate}
\end{thm}

\begin{proof}
For $r>0$, define  
\[ F_n^{(r)}(t,\omega,x) := \left\{ \begin{array}{ll} 
                         F_n(t,\omega, x)& \hbox{if } \|x\|_\theta \leq r\\
			 F_n\big(t,\omega,\frac{rx}{\|x\|_\theta}\big) & \hbox{otherwise}. \,
                        \end{array}
\right.
\]
For each $r>0$, the maps $F_n^{(r)}$ and $G_n^{(r)}$ 
are uniformly 
($\gamma$-)Lipschitz continuous and of linear growth. In particular, the processes $\bX_n^{(r)} := 
\sol (A_n, F_n^{(r)}, G_n^{(r)}, \xi_n)$ exist globally. Then the processes $\bX_n$ together with 
the processes $\bX_n^{(r)}$ satisfy the hypotheses of Theorem \ref{t.stopconv}.
 Indeed, (a) follows from 
the maximality of $\bX_n$, cf.\ \cite[Lemma 8.2]{vNVW08},
and (b) follows from the convergence $X_n\to X_\infty$ in $L^q(\Omega;C([0,T];E_\theta))$
of Proposition \ref{prop:global}.
\end{proof}

Concerning the initial data, it actually suffices to assume that
$\xi_n$ converges to  $\xi$ in $L^0(\Omega;\F_0,\P;E_\theta)$; see
Subsection \ref{sec:meas}.
 
\begin{cor}\label{c.main}
If, in addition to the assumptions of the previous theorem, we have  
$\sigma_n = T$ almost surely for all $n \in \CN$ and  
$\sup_{n\in\CN}\E \|\bX_n\|_{C([0,T]; E_\theta )}^p< \infty$,
then:
\begin{enumerate}
 \item $\sigma_\infty = T$ almost surely;
 \item For all $1\le q<p$, $$\bX_n \to \bX_\infty \ \hbox{ in }\, L^q(\Omega; C([0,T]; E_\theta ));$$
 \item For $0 \leq \delta < \half - \frac1p - \kappa_G$ we have 
$$\bX_n - S_n(\cdot )\xi_n \to \bX_\infty -S_\infty (\cdot )\xi_\infty \ \hbox{ in } \, 
 L^0((0,T)\times \Omega; E_\delta);$$
 \item 
If, in addition, \eqref{eq:exps2} holds and
$\sup_n\expect\|\bX_n - S_n(\cdot )\xi_n\|^p_{C^\lambda ([0,T]; E_\delta )} < \infty$, 
then 
\[ \bX_n - S_n(\cdot )\xi_n \to \bX_\infty -S_\infty (\cdot )\xi_\infty \ \hbox{ in }\,  
 L^q(\Omega; C^\mu ([0,T], E_\delta) )
\]
for all $1\leq q< p$ and $0\leq \mu < \lambda$.
\end{enumerate}
\end{cor}

\begin{proof}
(1) and (2) follow from Corollary \ref{c.hoelderconv}.\smallskip 

(3) Before we start the proof we note that the result follows trivially (with convergence in a stronger sense) 
from (2) when
$\delta\le \theta$. The point of (3) is that we might have $\delta>\theta$, and this is 
what we shall assume in the rest of the proof.

The processes $\bY_n := \bX_n - S_n(\cdot )\xi_n$
belong to $L^0(\Omega; C_{\rm b}([0,T); E_\delta ))$
in view of $\sigma_n = T$ and \cite[Theorem 8.1]{vNVW08}.
 
We first additionally assume that the initial values $\xi_n$ are uniformly bounded in 
$L^\infty (\Omega, \cF_0, \P; E_\theta)$ and put 
\[ C := \sup_{t \in [0,T], n \in \CN} \|S_n(t)\xi_n\|_{L^\infty (\Omega; E_\theta)}
 \, .
\]
Put ${Z}_n^{(r)} := \sol (A_n, F_n^{(r+C)}, G_n^{(r+C)}, \xi_n)$, where $F_n^{(r)}$ is as in the proof
of Theorem \ref{t.main}, and $\bY_n^{(r)} := {Z}_n^{(r)} - S_n(\cdot )\xi_n$.
If we put $\varrho_n^{(r)} := \inf\{ t > 0\, : \, \|Y_n(t)\|_\delta > r\}$, then 
$\bY_n\one_{[0,\varrho_n^{(r)}]} = \bY_n^{(r)}\one_{[0,\varrho_n^{(r)}]}$. 
Indeed, if $t \leq \varrho_n^{(r)}$, then
$\|Y_n(t)\|_\delta \leq r$ and
\[ \|Z_n^{(r)}(t)\|_\theta \leq \|Y_n^{(r)}(t)\|_\theta + \|S_n(t)\xi_n\|_\theta 
 \leq \|Y_n^{(r)}(t)\|_\delta + C \leq r +C
\]
almost surely.
By the maximality of $X_n$, $X_n \one_{[0,\varrho_n^{(r)}]} = {Z}_n^{(r)} \one_{[0,\varrho_n^{(r)}]}$. 
Subtracting $S_n(\cdot )\xi_n$, 
it follows that $\bY_n\one_{[0,\varrho_n^{(r)}]} = \bY_n^{(r)}\one_{[0,\varrho_n^{(r)}]}$ as claimed.

This proves that Hypothesis (a) preceding the statement of Theorem \ref{t.stopconv} is satisfied. Hypothesis (b) 
follows from Proposition \ref{prop:global}.
Thus the assertion follows from Theorem \ref{t.stopconv}(3).\smallskip 

It remains to remove the additional boundedness assumption. To that end, fix $K \in \CN$. 
From any given subsequence of 
$\xi_n$ we can extract a further subsequence, relabeled with indices $n$, 
such that $\xi_n\to\xi_\infty$ almost surely and 
$\|\xi_n - \xi_\infty\|_{L^p(\Omega; E_\theta )} \leq 2^{-n}$. 
By the Chebyshev inequality, $\P(\|\xi_n-\xi_\infty\|_\theta>1) \leq 2^{-np}$.

Now define $\Omega_K^N:=\{ \|\xi_n\|_\theta \leq K+1\, \forall \, n\geq N\}$. Then
\[ \P (\complement \Omega_K^N) \leq \P (\|\xi_\infty\|_\theta > K ) + 2^{-Np}\, .\]
Setting $\xi_n^{(K)} := \xi_n\one_{\{\|\xi_n\|_\theta \leq K+1\}}$, it follows that $\xi_n^{(K)} \to \xi_\infty^{(K)}$ 
in $L^p(\Omega; E_\theta )$ and $\xi_n^{(K)}$ is bounded in $L^\infty (\Omega; E_\theta )$. By the above, 
the claim holds true for the processes $\bY_n^{(K)}$, 
which are defined as the processes $\bY_n$, but starting the uncompensated solution 
at the modified initial data $\xi_n^{(K)}$. 

By \cite[Lemma 8.2]{vNVW08}, almost surely on $\Omega_K^N$, 
we have  $\bY_n^{(K)} = \bY_n$. Thus along our subsequence, (2) hold with $\Omega$ replaced with $\Omega_K^N$ for all 
$K, N \in \CN$. Writing $\Omega$ as a countable union of such sets, it follows that (2) holds as stated.\medskip 

(4) is immediate from (2) and \cite[Lemma 4.2]{KvN10}.
\end{proof}

\begin{example} The condition $\sup_{n\in\CN}\E \|\bX_n\|_{C([0,T]; E_\theta )}^p< \infty$
is satisfied if, in addition to the assumptions in Theorem \ref{t.main}, 
$F_n$ and $G_n$ are uniformly of linear growth. 
For $\lambda, \delta \geq 0$ 
with $\lambda +\delta <\half -\frac1p-\kappa_G$, we also have
$\sup_n \E\|\bX_n-S_n(\cdot )\xi_n\|_{C^\lambda([0,T]; E_\delta)}^p < \infty$; see 
\cite[Theorem 8.1]{vNVW08}. Hence, in this situation, Corollary \ref{c.hoelderconv}(4) applies.
\end{example}

\subsection{Stochastic evolution equations on general Banach spaces}\label{sec:Banach}

Reaction diffusion type equations with nonlinearities of polynomial growth are usually
considered in spaces of continuous 
functions. This is essential in order to verify the assumptions posed on the nonlinearities. 
As far as we know, there is no satisfying theory of 
stochastic integration available in spaces of continuous functions. We get around this by 
assuming that the Banach space $B$ in which we seek the solutions is sandwiched
between $E_\theta$ and $E$. We then assume that $E$ is a UMD Banach space
as in the previous section and carry out all stochastic integrations in the interpolation
scale of $E$. 
In order to be able to handle initial values with values
in $B$ without losing regularity due to the various embeddings, however, we need to carry out all
fixed point arguments in the space $L^p(\Omega;C([0,T];B))$. 

In applications, typical choices are 
$B = C(\overline{\OO})$ and $E = L^p(\OO)$ for some large $p\ge 2$, with $\OO$ a domain in $\R^d$.
This motivates us to work in UMD spaces $E$ with type $2$ from the onset 
(these include the spaces $L^p(\OO)$ for
$2\le p<\infty$). Accordingly we shall assume:
\begin{itemize}
\item[(E)] $E$ is a UMD Banach space with type $2$.
\end{itemize}
In addition to (A1) -- (A3) we shall assume:
\begin{itemize}
\item[(A4)] 
The semigroups $\bS_n$ restrict to strongly continuous 
semigroups $\bS_n^B$ on $B$ which are uniformly 
exponentially 
bounded in the sense 
that, for certain constants $\tilde{M} \geq 1$ and $\tilde{w} \in \CR$ we have 
$\|S_n(t)\|_{\cL (B)} \leq \tilde{M}e^{\tilde{w} t}$ for all $t\geq 0$ and $n \in \OCN$.
\item[(A5)]
We have continuous, dense embeddings
$E_\theta \inject B \inject E.$ 
\end{itemize}

Strong resolvent convergence of the parts $A_n|_B$ of $A_n$ in $B$ 
follows from (A1) -- (A4); see Lemma \ref{l.Bconv}.

In the applications we have in mind, the operators $A_n$ are second 
order elliptic differential operators on $E:= L^p(\OO)$ subject to 
suitable boundary conditions (b.c.), where $\OO\subseteq \CR^d$ is some domain,
and $E_\theta = H_{\rm b.c.}^{2\theta,p}(\OO)$ 
is the corresponding Sobolev space. 
If $p\ge 2$ and $\theta\ge 0$ are chosen appropriately in relation to the dimension $d$, 
then $E_\theta$ is continuously and densely embedded
into $B:= C_{\rm b.c.}(\overline{\OO})$. 

In the present framework we can repeat
the procedure of the previous subsection to 
obtain convergence to maximal solutions 
of \eqref{eq.scp} with nonlinearities $F$ and $G$ which are locally Lipschitz continuous
from a corresponding convergence result for globally Lipschitz continuous coefficients. 
In particular, the results of 
Theorem \ref{t.main} and Corollary \ref{c.hoelderconv} (1) and (2)
generalise {\em mutatis mutandis} to the situation considered here.
Instead of spelling out the details we content ourselves with the statement
of the convergence result for the globally Lipschitz case.

\begin{prop}\label{p.banachlip}
Let $B$ be a Banach space, assume (E) and (A1)--(A5), and assume that \eqref{eq:exps} holds with $\tau=2$, i.e.,
$2<p<\infty$, $0\le \theta<\frac12$, $0\le \kappa_G<\frac12$ satisfy
$$ \theta+\kappa_G < \tfrac12 -\tfrac1p.$$
Moreover, let $F_n : [0,T]\times\Omega\times B \to E_{-\kappa_F}$ and
$G_n : [0,T]\times\Omega\times B \to \gamma (H, E_{-\kappa_G})$ be strongly 
measurable, adapted, and globally 
Lipschitz continuous in the third variable, uniformly with respect to the first and second variables.
If
\[
 \limn F_n(t,\omega, x) = F_\infty (t,\omega, x)\quad\mbox{and}\quad \limn G_n(t,\omega, x) = G(t,\omega, x)
\]
for all $(t,\omega,x) \in [0,T]\times\Omega\times B$
and  $\xi_n \to \xi_\infty$ in  $L^p(\Omega, \cF_0, \P; B)$, then:
\begin{enumerate}
 \item For each $n \in \OCN$, the problem 
\eqref{eq.scp} with coefficients $(A_n,F_n,G_n)$ and initial datum $\xi_n$
has a unique mild solution $\bX_n$ in 
$L^p(\Omega; C([0,T]; B))$;
 \item For all 
$1\leq q < p$, $$\bX_n \to \bX_\infty \ \hbox{ in }\,L^q(\Omega; C([0,T]; B)).$$
\item If $\lambda, \delta \geq 0$ satisfy $\lambda + \delta < \half - \frac1p - \kappa_G$ 
then 
\[ X_n - S_n(\cdot )\xi_n \to X_\infty - S(\cdot )\xi_\infty \quad\mbox{in}\quad
 L^q(\Omega; C^{\lambda}([0,T]; E_\delta ))
\]
for all $1\leq q < p$.
\end{enumerate}
\end{prop}

Note that the condition 
$ \theta+\kappa_F<1$, which also results from \eqref{eq:exps} if we take $\tau=2$, 
is automatically satisfied in view of the standing assumptions $0\le \theta,\kappa_F<\frac12$.

\begin{proof}[Sketch of proof]
Towards (1),  let ${V_T}:=L^p_\mathbb{F}(\Omega; C([0,T]; B))$ denote the space of continuous, adapted $B$-valued processes $\phi$
such that $\|\phi\|^p_{V_T} := \expect \|\phi\|_{C([0,T];B)}^p < \infty$. 
By (A4), $S_n(\cdot )\xi_n \in {V_T}$. 

Consider the fixed point operators $\Lambda_{n,\xi_n,T}$  from ${V_T}$ into itself 
defined by
\[ \big[ \Lambda_{n,\xi_n,T} \phi \big](t) := S_n(t)\xi_n + \int_0^tS_n(t-s)F_n(s,\phi (s))\, ds 
 + \int_0^tS_n(t-s)G_n(s,\phi (s))\, ds.
\]
Using \cite[Lemma 3.4]{vNVW08}, we see 
that $\bS_n\ast F_n(\cdot, \phi )$ 
is in $L^p_\mathbb{F}(\Omega; C([0,T]; E_\theta))$, and hence in ${V_T}$, for all 
$\phi \in {V_T}$. 
Moreover, by the assumptions on 
$G_n$, we see that $s \mapsto S_n(t-s)G_n(s,\phi (s))$ 
is in $L^p(\Omega; L^2(0,t; \gamma (H,E_\theta )))$. Since $E_\theta$, 
being isomorphic to $E$, is UMD with type 2, 
this function is stochastically 
integrable in $E_\theta$. In fact, using 
\cite[Proposition 4.2]{vNVW08} one finds that the 
stochastic convolution 
defines an element of $L^p_\mathbb{F}(\Omega; C([0,T]; E_\theta))$, and hence of ${V_T}$.

Standard arguments show that for each $n$, $\Lambda_{n,\xi_n,T}$ is Lipschitz continuous on ${V_T}$ 
and the Lipschitz constants of $\Lambda_{n,\xi_n,T}$ 
converge to $0$ as $T\downarrow 0$. Hence, for small enough $T$, 
solutions of \eqref{eq.scp} can be obtained from Banach's 
fixed point theorem and global solutions of \eqref{eq.scp} can be 
`patched together' inductively from solutions on smaller time intervals.\medskip 

(2) As in the proof of \cite[Theorem 4.3]{KvN10} it suffices to prove that 
$\Lambda_{n,\xi_n,T}\phi \to \Lambda_{\infty, \xi_\infty,T}\phi$ in ${V_T}$ 
for all $\phi \in {V_T}$ with $T$ small.
Convergence of $S_n(\cdot )\xi_n \to S_\infty (\cdot )\xi_\infty$ follows from Lemma \ref{l.Bconv}. As for the 
stochastic and deterministic convolutions, as in \cite[Lemma 4.5]{KvN10} one sees 
that they actually converge in $L^p(\Omega; C([0,T]; E_\theta ))$, and hence 
in $L^p(\Omega; C([0,T]; B))$ by (A5).\medskip 

(3) follows similarly as in the proof of \cite[Theorem 4.7]{KvN10}.
\end{proof}

\subsection{Extension to measurable initial values}\label{sec:meas}
In our results so far, we have assumed that $\xi_n \to \xi_\infty$ in a suitable $L^p$-space. It is  
routine to extend these results to initial data $\xi_n$ which are merely assumed to be 
strongly measurable (but 
no integrability assumption is imposed) and convergent in measure. 

Given a strongly $\F_0$-measurable initial value $\xi$, 
a maximal solution of \eqref{eq.scp} 
is constructed as follows; cf.\ \cite[Section 7]{vNVW08}.
For $K \in \CN$, we put $\xi^{(K)}:= \xi\one_{\{\|\xi\|\leq K\}}$. If $(X^{(K)}(t))_{t\in [0,\sigma^{(K)})}$ denotes 
the maximal solution of \eqref{eq.scp} with initial datum $\xi^{(K)}$, then one sees 
that on the set $\{\|\xi\|\leq K\}$ we have $X^{(K)}\one_{\{t\in [0,\sigma^{(K)})\}} \equiv X^{K'}
\one_{\{t\in [0,\sigma^{(K')})\}}$
for all $K'\geq K$, see \cite[Lemma 8.2]{vNVW08}. Hence a maximal solution of \eqref{eq.scp} can be constructed 
from the solution $\bX^{(K)}$.

Now suppose that for $n \in \OCN$ we are given a measurable $\xi_n$ such that $\xi_n \to \xi_\infty$ in 
measure. Then for fixed $K$ the truncated random variables $\xi_n^{(K)}$ converge to $\xi_\infty^{(K)}$ in measure 
and are uniformly bounded, hence uniformly $p$-integrable for all $1\le p<\infty$. It follows 
that $\xi_n^{(K)} \to \xi_\infty^{(K)}$ in every $L^p$. Hence, cutting off the nonlinearities 
$F_n$ and $G_n$ as well as the initial data and arguing as before, we obtain convergence results for 
solutions of \eqref{eq.scp} with measurable initial data. We leave the details to the reader.

\section{Global existence for reaction diffusion type equations}\label{sect.3}

In this section, we shall make additional assumptions on the coefficients 
similar to those considered by Brze\'zniak and G{\c{a}}tarek \cite{bg99} and 
Cerrai \cite{Cerrai}.

Throughout this section we shall assume that $B$ is a Banach space and that $E$ is 
a UMD space with type $2$. {\em 
Unless explicitly stated otherwise all norms $\n \cdot\n$ are taken in $B$}.

Let us first recall 
that in a Banach space $B$, the \emph{subdifferential of the norm at $x$} is given by 
\[ \partial\|x\| := \big\{ x^* \in B^*\, : \, \|x^*\| =1\quad\mbox{and}\quad \langle x, x^*\rangle =1\big\}\, .\]
We recall, see \cite[Proposition D.4]{dpzab}, that if $u : I \to B$ is a differentiable function, then 
$\|u(\cdot )\|$ is differentiable from the right and from the left with 
\begin{align*}
\frac{d^+}{dt}\|u(t)\| & =   \max\big\{ \big\langle u'(t) ,x^*\big\rangle\,:\, x^*\in \partial\|u(t)\|\big\}, \\
\frac{d^-}{dt}\|u(t)\| & =  \min\big\{ \big\langle u'(t), x^*\big\rangle \,:\, x^*\in \partial\|u(t)\|\big\}.
\end{align*}
Since $\|u(\cdot)\|$ is everywhere differentiable from the left and from the right, it follows from 
\cite[Theorem 17.9]{hestro} that $\|u(\cdot )\|$ is differentiable, except for at most countably many points. 
In particular, $\|u(\cdot )\|$ is absolutely continuous and for almost all $t\in I$ we have 
\[
 \frac{d}{dt}\|u(t)\| = \big\langle u'(t), x^*\rangle \ \hbox{ for all }\, x^*\in \partial\|u(t)\|.
\]

Throughout this section the following standing assumptions will be in place.
We assume that $E$ is a UMD space with type $2$ and suppose 
that $A$ satisfies (A1), i.e., $A$ 
is the generator of a strongly continuous and analytic semigroup $\bS$ on $E$. 
Furthermore, we assume that (A4) and (A5) are satisfied and that $\bS^B$ 
is a strongly continuous contraction semigroup on $B$. In particular, 
$A|_B$ is dissipative. Concerning the maps $F$ and 
$G$ we make the following assumptions.

\begin{itemize}
 \item[(F$'$)] The map $F: [0,T]\times \Omega \times B \to B$ is 
locally Lipschitz continuous in the sense that
for all $r>0$, there exists a constant $L_F^{(r)}$ such that
\[ \|F(t,\omega, x) - F(t,\omega, y)\| \leq L^{(r)}_F\|x-y\|\]
for all $\|x\|,\|y\| \leq r$ and $(t,\omega ) \in [0,T]\times \Omega$ and there exists a 
constant $C_{F,0}\geq 0$ such that
\[ \|F(t,\omega, 0)\| \leq C_{F,0}\]
for all $t \in [0,T]$ and $\omega\in \Omega$.
Moreover, for all $x \in B$ the map $(t,\omega) \mapsto F(t,\omega, x)$ is strongly measurable and adapted.

For suitable constants $a',b'\ge 0$ and $N \geq 1$ we have
$$
 \langle Ax + F(t,x+y), x^*\rangle \leq a'(1+\|y\|)^N + b'\|x\|  
$$
for all $x \in \Dom(A|_B)$, $y \in B$, and $x^* \in \partial \|x\|$.
\item[(G$'$)] The map $G: [0,T]\times \Omega \times B \to \gamma (H, E_{-\kappa_G})$ is 
 locally Lipschitz continuous 
in the sense that for all $r>0$ there exists a constant $L_G^{(r)}$ such that
\[ \|G(t,\omega, x) - G(t,\omega, y)\|_{\gamma (H, E_{-\kappa_G})} \leq L_G^{(r)}\|x-y\| \]
for all $\|x\|, \|y\| \leq r$ and $(t,\omega ) \in [0,T]\times \Omega$ and there exists a 
constant $C_{G,0}\geq 0$ such that
\[ \|G(t,\omega, 0\|_{\gamma (H, E_{-\kappa_G})} \leq C_{F,0}\]
for all $t \in [0,T]$ and $\omega\in \Omega$. Moreover, for all $x \in B$ and $h \in H$ the map 
$(t,\omega )\mapsto G(t,\omega, x)h$ is strongly measurable and adapted. 

Finally, for suitable constants  
$c \geq 0$ and $\e>0$ we have
$$
\|G(t,\omega, x)\|_{\gamma (H, E_{-\kappa_G})} \leq c'(1+\|x\|)^{\frac{1}{N}+\e}  
$$
for all $(t,\omega, x) \in [0,T]\times \Omega \times B$.
\end{itemize}

\begin{rem}
In the results to follow, the constant $\e$ in (G$'$) has to be sufficiently small.
\end{rem}

\begin{example}\label{ex.f}
Let $B = C(\overline{\OO})$ for some bounded domain $\OO \subset \CR^d$.
Let $F:[0,T]\times \Omega\times B\to B$ be given by
$$(F(t,\omega,x))(s) = f(t,\omega,s,x(s)),$$
where 
\begin{equation}\label{eq:reactionterm}
f(t,\omega,s,\eta) = - a(t,\omega,s) \eta^{2k+1} + \sum_{j=0}^{2k} a_j(t,\omega,s)\eta^j, \quad \eta\in\CR.
\end{equation}
We assume that there are constants $0<c\le C<\infty$ such that (cf. \cite{bg99})
$$c \le a(t,\omega,s)\le C, \quad |a_j(t,\omega, s)|\le C \quad (j=0,\dots, 2k+1)$$
for all $(t,\omega, s) \in [0,T]\times\Omega\times\overline{\OO}$.
 It is easy to see that, in this situation, for a suitable constant $a'\ge 0$ we have
\[ -a'(1+|\eta|^{2k+1}\one_{\{\eta \geq 0\}} )\leq f(t,\omega, s, \eta) 
\leq a' (1 +|\eta|^{2k+1}\one_{\{u\leq 0\}}) 
\]
for all $t\in [0,T], \omega \in \Omega, s \in \OO, \eta \in \CR$.
This, in turn, yields that
\[ f(t,\omega, s, \eta+\zeta) \cdot\hbox{sgn}\, \eta \leq a'(1+|\zeta|^{2k+1}) \]
for all $(t,\omega, s) \in [0,T]\times\Omega\times\overline{\OO}$ and $\eta,\zeta\in \CR$.
By the results of \cite[Section 4.3]{dpz92} this implies 
$$\langle F(t,\omega, x+y), x^*\rangle \leq a'(1+\|y\|^{2k+1})$$
for all $t\in [0,T]$, $\omega\in\Omega$, $x,y \in B$, and $x^* \in \partial \|x\|$. 
Since $A|_B$ is dissipative,  it follows that (F$'$) holds.
 \end{example}

The first main result of this section reads as follows.

\begin{thm}\label{t.dissex}
Let $B$ be a Banach space, assume (E), (A1), (A4), (A5), (F$'$), and (G$'$) with $\e>0$ sufficiently 
small, and assume that 
$2<p<\infty$, $0\le \theta<\frac12$, $0\le \kappa_F, \kappa_G<\frac12$, and 
$$\theta +\kappa_G< \tfrac12 - \tfrac{1}{Np}.$$ 
For all $\xi \in L^p(\Omega, \cF_0, \P; B)$, the maximal solution 
$(X(t))_{t \in [0, \sigma)}$ of \eqref{eq.scp} is global, i.e., 
we have $\sigma = T$ almost surely. Moreover, 
\[ \expect\|\bX\|_{C([0,T]; B)}^p \leq C (1+ \expect \|\xi\|^p),\]
where the constant $C$ depends on the coefficients only through 
the sectoriality constants of $A$ and the constants $a',b',c'$ and the exponent $N$.
\end{thm}

This result improves corresponding results in \cite{bg99,Cerrai} 
under similar assumptions on $F$ and $G$. In \cite{bg99}, global existence of 
a martingale solution was obtained for uniformly bounded $G$; in \cite{Cerrai}, 
rather restrictive simultaneous diagonalisability assumptions on $A$ and the noise were imposed.

In the proof of Theorem \ref{t.dissex} we will use the following lemma, 
which is a straightforward generalisation of \cite[Lemma 4.2]{bg99}.
For the reader's convenience we include the short proof.

\begin{lem}\label{l.bg}
Let $A$ be the generator of a strongly continuous contraction semigroup $\bS$ on $B$, $x \in B$ and $F: [0,T]\times B \to B$
satisfy condition (F$'$). If for some $\tau >0$ two continuous functions $u, v: [0,\tau) \to B$ 
satisfy
\[ u(t) = S(t)x + \int_0^t S(t-s)F(s, u(s) + v(s))\, ds \quad \forall\, t \in [0,\tau),\]
then 
\[ \|u(t)\| \leq e^{b't}\Big( \|x\| + \int_0^t a'(1+\|v(s)\|)^N\, ds\Big)\, .\]
\end{lem}

\begin{proof}
For $n \in \CN$, put $u_n(t) := nR(n,A)u(t)$, $x_n := nR(n,A)x$ 
and $F_n(t, y) := nR(n,A)F(t,y)$. Then
\[ u_n(t) = S(t)x_n + \int_0^tS(t-s)\big[ F(s,u_n(s)+v(s)) + r_n(s)\big]\, ds \]
where $r_n(s) = F_n(s,u(s)+v(s)) - F(s, u_n(s) + v(s))$. It follows that $u_n$ is differentiable with
\[ u_n'(t) = Au_n(t) + F(t,u_n(t) + v(t)) + r_n(t)\, .\]
By the observations at the beginning of this section, for almost all $t\in (0,T)$ we have,
for all $x^* \in \partial\|u_n(t)\|$,
\[ 
\begin{aligned}
\frac{d}{dt}\|u_n(t)\| & = \big\lb A(t)u_n(t)+ F(s,u_n(t)+ v(t)) + r_n(t),x^*\big\rb\\
 & \leq a'(1+\|v(t)\|)^N + b'\|u_n(t)\| + \|r_n(t)\|.
\end{aligned}
\]
Thus, by Gronwall's lemma,
\[ \|u_n(t)\| \leq e^{b't}\Big( \|x_n\| + \int_0^t a'(1+\|v(s)\|)^N+ \|r_n(s)\|\, ds\Big)\, .\]
Since $\|nR(n,A)\|\leq 1$ and $nR(n,A)\to I$ strongly as $n\to\infty$, 
the assertion follows
upon letting $n\to \infty$.
\end{proof}

\begin{proof}[Proof of Theorem \ref{t.dissex}] 
Let us first assume that (G$'$) is satisfied with $\e=0$; in the proof we indicate
the reason why a small $\e>0$ can be allowed.

We define 
\[ F_n(t,\omega, x) := \left\{\begin{array}{ll}
                               F(t,\omega, x) & \hbox{if } \|x\|\leq n,\\
			       F\big(t,\omega, \frac{nx}{\|x\|}\big) & \mbox{otherwise} .
                              \end{array}
 \right.
\]
We also set $\bX_n := \sol (A, F_n, G, \xi)$. 

Let us first note that for $x \in \Dom(A|_B)$, $y \in B$, and $x^* \in \partial \|x\|$, we have
\begin{equation}\label{eq.dissipative}
 \langle Ax + F_n(t, \omega,x+y), x^*\rangle \leq a'(1+\|y\|)^N + b\|x\|\,  
\end{equation}
for all $(t,\omega ) \in [0,T]\times\Omega$.
If $\|x+y\|\leq n$, then this follows directly from (F$'$). If $\|x+y\|>n$, then
\begin{align*}
 \big\langle Ax + F_n(t,\omega, x+y), x^*\big\rangle & =  
\Big\langle Ax + F\big(t,\omega, \frac{nx}{\|x+y\|} + \frac{ny}{\|x+y\|}\Big),
x^*\big\rangle\\
 & =  \Big\langle A \frac{nx}{\|x+y\|}+ F\big(t,\omega, \frac{nx}{\|x+y\|} + \frac{ny}{\|x+y\|}\big),
x^*\Big\rangle\\
&\quad + (1-\frac{n}{\|x+y\|})\langle Ax,x^*\rangle\\
& \leq  a'\Big( 1 + \Big(\frac{n\|y\|}{\|x+y\|}\Big)\Big)^N + b'\Big\|\frac{n\|y\|}{\|x+y\|}\Big\|\\
& \leq  a' (1+\|y\|^N) + b'\|x\|
\end{align*}
where we have used (F$'$) and the dissipativity of $A$ in $B$ in the third 
step and $\|x+y\|>n$ in the fourth.

Trivially,
\begin{equation}\label{eq:bXn}
\begin{aligned}
\ &  \expect\|\bX_n\|_{C([0,T]; B)}^p 
\\ & \hskip1.4cm  \lesssim  \expect \|S(\cdot )\xi + \bS\ast F_n(\cdot, \bX_n)\|_{C([0,T]; B)}^p
+ \expect \|\bS \diamond G(\cdot, \bX_n)\|_{C([0,T]; B)}^p \, .
\end{aligned}
\end{equation}
By \eqref{eq.dissipative} and Lemma \ref{l.bg}, applied with 
\[ u_n := \bX_n - \bS\diamond G(\cdot, \bX_n),\quad v_n =  \bS \diamond G(\cdot, \bX_n ), \]
we obtain
\begin{align*}
 & \expect \|S(\cdot )\xi + \bS\ast F_n(\cdot, \bX_n)\|_{C([0,T]; B)}^p
\\ & \hskip1.4cm =
\expect \sup_{t \in [0,T]} \Big\n S(\cdot )\xi + \int_0^t S(t-s)F_n(s,u_n(s)+v_n(s))\, ds\Big\n^p 
\\ & \hskip1.4cm \leq  e^{b'pT}
\expect \sup_{t \in [0,T]} \Big(\|\xi\| + \int_0^t a'\big( 1 +\|v_n(s) \|\big) ^N\, ds 
\Big)^p\\ 
& \hskip1.4cm \lesssim  e^{b'pT}\expect
\big(1+ \|\xi\|^p+ \|\bS\diamond G(\cdot , \bX_n)\|_{C([0,T]; B)}^N\big)^p\\
& \hskip1.4cm \lesssim  e^{b'pT}T^p\big( 1+ \expect\|\xi\|^{p} + \expect\|\bS\diamond G(\cdot , \bX_n)\|_{C([0,T]; B)}^{Np}\big)\, .
\end{align*}
Since $\theta + \kappa_G < \half -\frac{1}{Np}$, we may pick $\alpha \in (0,\half )$ such that
$\theta + \kappa_G < \alpha - \frac{1}{Np}$. Then, for some $\e>0$,
\begin{align*}
\ &  \expect\|\bS\diamond G(\cdot , \bX_n)\|_{C([0,T]; B)}^{Np}
\\ & \hskip1.4 cm  \lesssim \expect\|\bS\diamond G(\cdot , \bX_n)\|_{C([0,T]; E_\theta)}^{Np}
\\ & \hskip1.4 cm \lesssim T^{\e Np} \E \int_0^T \|s \mapsto (t-s)^{-\alpha} 
G(\cdot, \bX_n)\|_{\gamma (L^2(0,t; H), E_{-\kappa_G})}^{Np}\, dt
\\ & \hskip1.4 cm \lesssim T^{\e Np} \E \int_0^T  \|s \mapsto (t-s)^{-\alpha} G(\cdot, \bX_n)
\|_{L^2(0,t;\gamma( H, E_{-\kappa_G}))}^{Np}\, dt
\\ & \hskip1.4 cm = T^{\e Np} \E \int_0^T \Big( \int_0^t (t-s)^{-2\alpha}\|G(s, X_n(s))
\|_{\gamma (H, E_{-\kappa_G})}^{2}\, ds\Big)^{\frac{Np}{2}}\, dt
\\ & \hskip1.4 cm \stackrel{(*)}{\leq} T^{\e Np} \Big(\int_0^T t^{-2\alpha} \, dt\Big)^{\frac{Np}{2}} \expect 
\int_0^T \|G(t, X_n(t))\|_{\gamma (H, E_{-\kappa_G})}^{Np}\, dt
\\ & \hskip1.4 cm \le T^{(\frac12-\alpha+\e)Np} (c')^{Np} 
 \expect \int_0^T (1+\|X_n(t)\|)^{p}\, dt
\\ & \hskip1.4 cm \lesssim T^{(\frac12-\alpha+\e)Np+1}(c')^{Np} 
(1 +  \expect \|\bX_n\|_{C([0,T]; B)}^p)\, .
\end{align*}
In this computation we used the following facts. 
The first inequality follows from the continuity of the embedding
$E_\theta\embed B$, the second 
from \cite[Proposition 4.2]{vNVW08} (here the condition on $\alpha$ is used),
the third uses the fact that if $(S,\mu)$ is a $\sigma$-finite measure space,
$H$ a Hilbert space and $F$ a Banach space 
with type 2, then we have a continuous embedding $L^2(S,\mu;\gamma(H,F))\embed 
\g(L^2(S,\mu;H),F)$ of norm less than or equal to the type $2$ constant of $F$,
in the next inequality we used Young's inequality, and in the sixth step the assumptions on $G$.

Because of the strict inequality $\a<\frac12$, in step $(*)$ we can apply Young's inequality
with slightly sharper exponents. This creates room (explicitly computable in terms of the other
exponents involved) for a small $\e>0$ in Hypothesis (G$'$). 

Combining these estimates we obtain
\begin{align*}
\expect \|S(\cdot )\xi +& \bS\ast F_n(\cdot, \bX_n)\|_{C([0,T]; B)}^p\\
& \lesssim e^{b'pT}T^p
\big(1 + \expect\|\xi\|^p + T^{(\frac12-\alpha+\e)Np+1} (1+ \expect \|\bX_n\|_{C([0,T]; B)}^p)\big)\, .
\end{align*}
Next, 
\begin{align*} 
\expect \|\bS \diamond G(\cdot, \bX_n)\|_{C([0,T]; B)}^p 
 & \leq \big(\expect \|\bS 
\diamond G(\cdot, \bX_n)\|_{C([0,T]; B)}^{Np}\big)^{\frac{1}{N}}
\\ & \lesssim T^{(\frac12-\alpha+\e)Np+1}(1+\expect \|\bX_n\|_{C([0,T]; B)}^p)\, .
\end{align*}
Substituting these estimates into \eqref{eq:bXn} we obtain
\[ \expect\|\bX_n\|_{C([0,T]; B)}^p \leq C_0+C_1 \expect \|\xi\|^{p} + C_2(T) \expect\|\bX_n\|_{C([0,T]; B)}^p \]
for a certain constants $C_0$, $C_1$ and a function $C_2(T)$ which does not depend on $\xi$ and
converges to $0$ as $T\downarrow 0$.
Hence, if $T>0$ is small enough,
we obtain 
$$\expect\|\bX_n\|_{C([0,T]; B)}^p \leq (1-C_2(T))^{-1}(C_0+ C_1\expect \|\xi\|^{p}).$$ 

Iterating this procedure a finite number of times, 
it follows that given $T>0$, there exists a 
constant $C$ as in the statement such that
$\sup_n \expect\|\bX_n\|_{C([0,T]; B)}^p \leq C(1+\expect\|\xi\|^p) < \infty$.
By Corollary \ref{c.hoelderconv}, the lifetime of $X$ equals $T$ almost 
surely and we have
$\bX_n \to \bX$ in $L^q(\Omega; 
C([0,T]; B))$ for all $1\leq q< p$.
\end{proof}

Our next aim is to prove a version of 
Theorem \ref{t.dissex} (Theorem \ref{t.dissex2} below) which, in return for an additional
assumption on $F$, allows
nonlinearities $G$ of linear growth. For this purpose we introduce the following hypotheses.

\begin{itemize}
 \item[(F$''$)] There exist constants $a'',b'',m>0$ such that the 
function $F: [0,T]\times \Omega \times B \to B$ 
satisfies
$$
\ \qquad \lb F(t,\omega, y+x) - F(t,\omega, y),x^*\rb \leq  a''(1+\|y\|)^m -b''\|x\|^m
$$
for all $t \in [0,T]$, $\omega \in \Omega$, $x,y \in B$, and $x^* \in \partial \|x\|$,
and 
\[ \|F(t,y)\| \leq a''(1+\|y\|)^m \]
for all $y \in B$. 

\item[(G$''$)]
The function 
$G: [0,T]\times\Omega\times B \to \gamma (H, E_{-\kappa_G})$ satisfies the measurability and 
adaptedness assumption of (G$'$) and is 
locally Lipschitz continuous 
and of linear growth. Moreover, we have $$\|G(t,\omega, 0)\|_{\gamma (H, E_{-\kappa_G})} \leq C_{G,0}$$ for all 
$(t,\omega) \in [0,T]\times \Omega$ and a suitable constant $C_{G,0}\geq 0$. 
\end{itemize}

\begin{example}\label{ex.f2}
The map $F$ described in Example \ref{ex.f} also satisfies condition (F$''$). Indeed, for the function 
$f$ as in Example \ref{ex.f}, it is easy to see that for certain constants $a_1, a_2 \in \CR$ and $b_1,b_2>0$
we have
\[a_1 - b_1\eta^{2k+1} \leq f(t,\omega, s, \eta ) \leq a_2 - b_2\eta^{2k+1} \]
for all $(t,\omega, s, \eta ) \in [0,T]\times\Omega\times \overline{\OO}\times \CR$. But this 
yields that
\begin{equation}\label{eq.toshow}
 W:= \big[ f(t,\omega, s, \eta+\zeta ) - f(\zeta )\big]\cdot\hbox{sgn}\,\eta \leq a - b|\eta|^{2k+1} + c|\zeta|^{2k+1}
\end{equation}
for certain positive constants $a,b,c$ and all $(t,\omega, s) \in [0,T]\times \Omega \times \overline{\OO}$ 
and $\eta, \zeta \in \CR$.\smallskip 

To see this, we distinguish several cases.
\begin{itemize}
 \item $\eta,\zeta \geq 0$. In this case,
\begin{align*}
 W & \leq  a_2 - b_2(\eta +\zeta )^{2k+1} - a_1 + b_1\zeta^{2k+1}\\
& \leq  a_2 - a_1 - b_2|\eta|^{2k+1} + b_1|\zeta|^{2k+1}
\end{align*}
since $\eta +\zeta \geq \eta = |\eta|$.
\item $\eta, \zeta \leq 0$. In this case, 
\begin{align*}
 W & \leq  a_2 - b_2\zeta^{2k+1} - a_1 + b_1(\eta+\zeta)^{2k+1}\\
 &= a_2 - a_1 + b_2|\zeta|^{2k+1} - b_2(|\eta | + |\zeta|)^{2k+1}\\
& \leq  a_2-a_1 + b_2|\zeta|^{2k+1} - b_2|\eta|^{2k+1}.
\end{align*}
\item $\eta \leq 0 \leq \zeta$. In this case,
\begin{align*}
 W & \leq  a_2 - b_2\zeta^{2k+1} - a_1 + b_1(\eta+ \zeta)^{2k+1}\\
& = a_2 - b_2|\zeta|^{2k+1} - a_1 + b_1(|\zeta| - |\eta|)^{2k+1}. 
\end{align*}
If $|\zeta|\geq |\eta|$, then this can be estimated by
\[ a_2 -a_1 - b_2|\eta|^{2k+1} + b_1|\zeta|^{2k+1}\, .\]
If $0\neq |\zeta| \leq |\eta|$, then
\begin{align*}
W & \leq  a_2-a_1 + b_1|\zeta|^{2k+1}\big(1-\big|\frac{\eta}{\zeta}\big|\big)^{2k+1}\\
& \leq a_2-a_1 + b_1|\zeta|^{2k+1}\Bigg(1 - \big|\frac{\eta}{\zeta}\big|^{2k+1} 
+ \sum_{j=1}^{2k}\binom{2k+1}{l}\Bigg)\, . 
\end{align*}
\item The case where $\zeta \leq 0 \leq \eta$ can be handled similarly.
\end{itemize}
This shows that \eqref{eq.toshow} holds for $a= a_2 -a_1  ,b = \min\{b_1,b_2\}$
and $c=\max\{b_1,b_2\}\big(1+ \sum_{j=1}^{2k}\binom{2k+1}{l}\big)$. Now, with the same strategy as 
in \cite{dpz92}, one infers (F$''$) from \eqref{eq.toshow}.
\end{example}

Following the ideas of \cite{Cerrai}, we proceed through the use of a comparison principle. 
For the reader's convenience we include the proof, which is similar to that of 
\cite[\S9 Satz IX]{walter}.
\begin{lem}\label{l.comp}
Let $f: (a,b)\times (c,d)\to \CR$ be continuous and uniformly locally Lipschitz continuous in the second 
variable, i.e., for all compact $K \subseteq (c,d)$ there exists a constant $L=L(K)$ such that
\[ |f(t,x) - f(t,y)|\leq L|x-y|\quad\forall\, x,y \in K, \ t \in (a,b)\, .\]
Suppose the functions $u^+, u^- :[\alpha, \beta]\to (c,d)$
are absolutely continuous functions and satisfy, for almost all $t \in (\alpha, \beta)$,  
\[ \frac{d}{dt}u^+(t) \geq f(t, u^+(t)), \quad \frac{d}{dt}u^-(t) \leq f(t,u^-(t)).  \]
If $u^+(t_0)>u^-(t_0)$ for some $t_0 \in [\alpha, \beta]$, 
then $u^+(t) > u^-(t)$ for all $t \in [t_0,\beta ]$.
\end{lem}
 
\begin{proof}
We may of course assume that $t_0\in [\a,\b)$. 

Put $d(t) := u^+(t)-u^-(t)$ for $t \in [\alpha, \beta]$. Suppose that 
$A:= \{ t \in (t_0, \beta] : d(t)\leq 0\}$
is nonempty. Then, by continuity, $t_1:= \inf A > t_0$. Moreover, $d(t) >0$ on $[t_0, t_1)$ 
and $d(t_1) = 0$.

Let $K = K^+\cup K^-$ with $K^\pm := \{ u^\pm (t):\, t \in [\alpha, \beta]\}$
and denote by $L$ the corresponding 
Lipschitz constant from the hypothesis. For almost all $s \in (t_0, t_1)$ we have
\[
\begin{aligned} 
d'(s) &= \frac{d}{ds} (u^+(s) - u^-(s))
 \geq f(s, u^+(s) ) - f(s, u^-(s))\\
& \geq -L|u^+(s) - u^-(s)| 
= -Ld(s)
\end{aligned}
\]
since $s<t_1$. 
It follows that $\frac{d'}{d} \geq -L$ almost everywhere on $(t_0,t_1)$ and hence, by integration, 
$d(t) \geq d(t_0)e^{-L(t-t_0)}$ for all $t \in (t_0, t_1)$. 
By continuity, $d(t_1) \geq d(t_0)e^{-L(t_1-t_0)}>0$, which contradicts $d(t_1)=0$. 
Hence we must have $A=\emptyset$ and thus $u^+(t) > u^-(t)$ for all $t \in (t_0, \beta]$ as claimed. 
\end{proof}

\begin{cor}\label{c.comp}
Let $f$ and $u^+, u^-$ be as in Lemma \ref{l.comp} but assume now that $u^+(t_0) \leq u^-(t_0)$ for 
some $t_0 \in [\alpha, \beta]$. Then $u^+(t) \leq u^-(t)$ for all $t \in [\alpha, t_0]$.
\end{cor}

\begin{proof}
If $u^+(t_1) >u^-(t_1)$ for some $t_1 \in [\alpha, t_0)$, then Lemma \ref{l.comp} would imply that 
$u(t_0) > u^-(t_0)$. 
\end{proof}

The next lemma should be compared with Lemma \ref{l.bg}.

\begin{lem}\label{l.dissbound}
Let $A$ be the generator of a strongly continuous contraction semigroup $\bS$ on $B$ 
and let $F: [0,T]\times B\to B$ satisfy conditions (F$'$) and (F$''$).
If $u, v \in C([0,T]; B)$ satisfy 
\[ u(t) = \int_0^t S(t-s)F(s, u(s) + v(s))\, ds,\]
for all $t\in [0,T]$, then
$$\sup_{t \in [0,T]}\|u(t)\| \leq \big(\frac{4a''}{b''}\big)^{\frac{1}{m}}
\Big(1+\sup_{t \in [0,T]}\|v(t)\|\Big).$$
\end{lem}

\begin{proof}
To simplify notations we write $a=a''$ and $b=b''$, where 
$a''$, $b''$ are as in (F$''$).

{\em Step 1} -- First we assume that $A$ is bounded. Then $u$ is continuously differentiable and 
\[ u'(t) = Au(t) + F(t, u(t)+v(t)), \quad t \in [0,T]. \]
By the remarks at the beginning of the section, for almost all $t\in (0,T)$ we have,
for all $x^* \in \partial\|u(t)\|$,
\[
 \begin{aligned}
  \frac{d}{dt}\|u(t)\| & = \lb Au(t), x^*\rb + \lb F(t,u(t)+v(t)) -F(t, v(t)),x^*\rb + \lb F(t, v(t)),x^*\rb\\
& \leq  0 + 2a(1+\|v(t)\|)^m  -b\|u(t)\|^m\,  
\\ & \leq  2a\Big(1+ \sup_{s \in [0,T]}\|v(s)\|\Big)^m -b\|u(t)\|^m \, .
 \end{aligned}
\]
In the second estimate we have used the dissipativity of $A$ and our assumptions.

Setting $\varphi (t) := \|u(t)\|$ and 
$\gamma := (2a)^{\frac{1}{m}}(1+ \sup_{s \in [0,T]}\|v(s)\|)$, it follows 
that $\varphi$ is absolutely continuous with 
\[ \varphi'(t) \leq -b\varphi (t)^m + \gamma^m \]
almost everywhere. We have to prove that $\varphi (t) \leq \big(\frac{2}{b}\big)^\frac{1}{m}\gamma$ 
for all 
$t \in [0,T]$. Assume to the contrary that $\varphi (t_0) 
> \big( \frac{2}{b}\big)^\frac{1}{m}\gamma$ for 
some $t_0 \in [0,T]$. Clearly $\varphi(0)=0$, so $t_0\in (0,T]$. 
Let $\psi : I \to \CR$ be the unique maximal solution of 
\[ \left\{ 
\begin{aligned}
            \psi'(t) & =  -b\psi(t)^m + \gamma^m,\\
	    \psi(t_0) & =  \varphi(t_0).
           \end{aligned}
 \right. 
\] 
By Corollary \ref{c.comp}, $\psi (t) \leq \varphi (t)$ for all $t\in I\cap [0,t_0]$.

We claim that $\psi (t) > (\frac1{b})^{\frac{1}{m}}\gamma$ for all $t \in I\cap [0,t_0]$. 
If the claim was false, noting that $\psi(t_0) = \varphi(t_0) > \big( \frac{2}{b}\big)^\frac{1}{m}\gamma$,
we would have $\psi (t_1) = \big( \frac{1}{b}\big)^\frac{1}{m}\gamma$ for some $t_1 \in I\cap [0,t_0]$.
By uniqueness, this would 
imply that $\psi \equiv (\frac1{b})^{\frac{1}{m}}\gamma$, a contradiction to 
$\psi (t_0) > (\frac1{b})^{\frac{1}{m}}\gamma$. This proves the claim.

We have proved that $\big( \frac{1}{b}\big)^\frac{1}{m}\gamma  < \psi\le \varphi$ on $I\cap [0,t_0]$. 
It follows that $0 \in I$ since otherwise $\psi$, and hence $\varphi$, would blow up
at some point in $[0, t_0)$. 

Consequently, $\big( \frac{1}{b}\big)^\frac{1}{m}\gamma  < \psi$ on $I\cap [0,t_0]$, 
which implies that $\psi'(t) < 0$ 
and hence that $\psi$ is decreasing. It follows that
\[ 0 = \varphi (0) \geq \psi (0) \geq \psi (t_0) = \varphi (t_0) 
 > \big( \frac{2}{b}\big)^\frac{1}{m}\gamma ,
\]
a contradiction.\medskip 

{\em Step 2} -- In order to remove the assumption that $A$ is bounded, we approximate $A$ with its 
Yosida approximands $A_n := nAR(n,A)= n^2R(n,A) -n$. We note that if $A$ is dissipative, then so are all $A_n$. 
We denote the (contraction) semigroup generated $A_n$ by $\bS_n$. Let $u_n$ be the unique 
fixed point in $C([0,T]; B)$ of 
\[ w \mapsto \Big[ t \mapsto \int_0^t S_n(t-s)F(s, w(s)+v(s))\, ds \Big]\, .\]
We note that, by the local Lipschitz assumption on $F$, there always exists a unique maximal solution 
of this equation. By Theorem \ref{t.dissex} with $G\equiv 0$ this solution is global.
Assumption (A3) is not needed for this part of the argument; cf. Remark \ref{rem:A3}.

By the above, 
\[ \sup_{t \in [0,T]}\|u_n(t)\| \leq \big(\frac{4a}{b}\big)^{\frac{1}{m}}
\Big(1+\sup_{t \in [0,T]}\|v(t)\|\Big)\]
for all $n \in \CN$. Since $u_n \to u$ in $C([0,T]; B)$, this gives the desired result.
\end{proof}

We can now extend Theorem \ref{t.dissex} assuming that $G$ is of linear growth.

\begin{thm}\label{t.dissex2} 
Assume (A1), (A4), (A5), (F$'$), (F$''$), (G$''$) and 
let $p>2$ satisfy $\theta +\kappa_G< \half - \frac{1}{p}$. \footnote{In the published version of the paper (J. Differential Equations 253 (2012), no. 3, 1036--1068), it was incorrectly stated that $p>2$ should satisfy $\theta +\kappa_G< \half - \frac{1}{Np}$. We thank Carlo Marinelli for kindly pointing this out.}
Then for all $\xi \in L^p(\Omega, \cF_0, \P; B)$ the maximal solution 
$(X (t))_{t \in [0, \sigma)}$ of \eqref{eq.scp} is global. Moreover, 
\[ \expect\|\bX\|_{C([0,T]; B)}^p \leq C (1+ \expect \|\xi\|^p),\]
where the constant $C$ depends on the coefficients only through 
the sectoriality constants of $A$ and the constants $a'',b'',c''$ and the exponent $N$.
\end{thm}

\begin{proof}
For $n \in \CN$ we put
 \[ G_n(t,\omega, x) := \left\{\begin{array}{lcl}
                               G(t,\omega, x), & \|x\|\leq n\\
			       G\big(t,\omega, \frac{nx}{\|x\|}\big), & \mbox{otherwise} .
                              \end{array}
 \right.
\]
Since $G$ is of linear growth, $G_n$ is bounded. In particular, 
$A, F$ and $G_n$ satisfy the
Hypotheses (F$'$) and (G$'$) uniformly with respect to $n$. Hence, by Theorem \ref{t.dissex}, $\bX_n := \sol (A, F, G_n, \xi )$ exists 
globally.

Proceeding as in the proof of Theorem \ref{t.dissex}, we have
\begin{equation}\label{eq.sum}
 \expect\|X_n\|_{C([0,T]; B)}^p \lesssim \expect \|\xi\|^p + \expect \|\bS\ast F(\cdot , \bX_n)\|_{C([0,T];B)}^p
+ \expect \|\bS\diamond G_n(\cdot, \bX_n)\|_{C([0,T]; B)}^p\, .
\end{equation}
Using Lemma \ref{l.dissbound} with 
\[ u_n = \bX_n - S(\cdot )\xi - \bS\diamond G(\cdot, \bX_n)\quad\mbox{and}\quad
 v_n = S(\cdot )\xi + \bS\diamond G(\cdot, \bX_n)
\]
we obtain 
\[
 \expect \|\bS\ast F(\cdot , \bX_n)\|_{C([0,T];B)}^p
\lesssim  \big(\frac{4a''}{b''}\big)^\frac{1}{m}(1 + \expect\|\xi\|^p + \expect\|\bS\diamond G(\cdot , \bX_n)\|_{C([0,T];B)}^p)
\, . 
\]
Moreover, a computation similar as in the proof of Theorem \ref{t.dissex} yields
\[ 
 \expect \|\bS\diamond G_n(\cdot, \bX_n)\|_{C([0,T]; B)}^p \leq 
C(T)\big( \alpha + \beta \expect\|\bX_n\|_{C([0,T];B)}^p\big)
\]
where $C(T)\to 0$ as $T\to 0$ and $\alpha, \beta$ only depend on the constants in the linear growth assumption on 
$G$. Substituting this back into \eqref{eq.sum}, it follows that
\[ 
  \expect\|X_n\|_{C([0,T]; B)}^p \leq C_0 + C_1\expect \|\xi\|^p + C_2(T)  \expect\|X_n\|_{C([0,T]; B)}^p
\]
and the proof can be finished as that of Theorem \ref{t.dissex} 
\end{proof}

In combination with our earlier results, it can be seen that the solution $\bX$ in Theorem \ref{t.dissex2}
depends continuously on the data $A$, $F$, $G$, and $\xi$ in the sense discussed in Section \ref{sect.2}.
We leave the precise statement of this result to the reader.

%

\section{1D Reaction diffusion equations with space-time white noise}\label{sect.4}

In this section, we apply our results to stochastic reaction diffusion equations in dimension 1 with multiplicative 
space-time white noise. The restriction to the 1 dimensional situation is made so that space-time white noise 
can be considered.
However, it is possible to extend our results also to higher space dimensions, provided that (i) white noise 
is replaced with either coloured or finite-dimensional noise (ii) appropriate regularity assumptions are made on 
the domain. For ease of notation, we will also only consider coefficients $f$ and $g$ which do not depend on $\omega$.

We consider the stochastic partial differential equation
\begin{equation}\label{eq.spde}
\left\{ 
\begin{aligned}
\frac{\partial}{\partial t} u(t,x) & =  \A u (t,x) + f(t,x, u(t,x)) + g(t,x,u(t,x))\frac{\partial w}{\partial t}(t,x),\\
u(0,x) & =  \xi (x),
\end{aligned}
\right. 
\end{equation}
for $(t,x)\in [0,T]\times [0,1] $.
We complement \eqref{eq.spde} with either Dirichlet or Neumann boundary conditions.
Here, $\A$ is a second order elliptic operator, formally given by
\[ \A u = (au')' + bu' +c \]
where the coefficients $a,b,c$ are real-valued and belong to $L^\infty (0,1)$, and 
$$\inf_{x \in (0,1)}a(x) > 0.$$
 The nonlinearity $f: [0,T]\times 
[0,1]\times \CR\to \R$ is as in Example \ref{ex.f} and $g: [0,T]\times [0,1]\times \CR\to \R$ is locally Lipschitz continuous 
and of linear growth in the third variable, uniformly with respect to the first two variables. Finally, 
$w$ is a space-time white noise on $[0,1]$.

Let us rewrite equation \eqref{eq.spde} in our general abstract framework. We set 
$$H = L^2(0,1), \quad E = L^q(0,1)$$
with $2\le q<\infty$. Then $E$ is a UMD Banach space with type 2.
On $H$ we consider the closed sectorial form
\[ \mathfrak{a}[u,v] := \int_0^t a(x)u'(x)\overline{v'(x)} - b(x)u'(x)\overline{v(x)} - c(x)u(x)\overline{v(x)}\, dx \]
defined on the domain
$ \mathsf{D}(\mathfrak{a}) := H^1_0(0,1)$ or $\mathsf{D}(\mathfrak{a}):= H^1(0,1)$, depending on whether 
we impose Dirichlet or Neumann boundary conditions. 
The operator $(A, \mathsf{D}(A))$ is defined as usual by
defining $ \mathsf{D}(A)$ as the subspace of all 
$u \in \mathsf{D}(\mathfrak{a})$ such that 
there exists a unique element $w \in H$ such that 
\[ -[w,v]_H = \mathfrak{a}[u,v]\quad \forall\, v \in D(\mathfrak{a}). \]
We then set $Au := w$ for $u\in  \mathsf{D}(A)$.
It is well known (see \cite{ouhabaz}) that $A$ generates a strongly continuous, 
analytic semigroup $\bS$ on $L^2(0,1)$. Moreover, 
$\bS$ is positive and, replacing $\mathfrak{a}$ with $\mathfrak{a} - c_0$ for some 
$c_0 \ge 0$ if necessary, 
we may assume that $\bS$ is sub-Markovian. In particular, $S$ extends to a 
strongly continuous semigroup $S^E$ on $E = L^q(0,1)$, and, as a consequence of 
the Stein interpolation theorem (see \cite[Proposition 3.12]{ouhabaz}), 
this semigroup is analytic on $E$.
By rescaling we may assume that $\bS^E$ 
is uniformly exponentially stable on $E$ (we may replace $A^E$ and $f$ with $A^E-c_0$ and 
$f+c_0$).

Since $\mathsf{D}(A) \subseteq \mathsf{D}(\mathfrak{a}) \subseteq C([0,1])$ and 
$S(t)f \in \mathsf{D}(A)$ for all $f \in E$ and $t>0$ by the analyticity of $\bS$, it follows that $C([0,1])$ is invariant under $\bS$. Moreover,
if $B$ denotes the closure of $D(\mathfrak{a})$ in $C([0,1])$ 
(i.e., $B= C_0([0,1])$ in the case of Dirichlet 
boundary conditions and $B= C([0,1])$ in the case of Neumann boundary conditions) then $\bS$ 
restricts to a strongly 
continuous contraction semigroup $\bS^B$ on $B$.

The next lemma identifies the associated fractional domain spaces.

\begin{lem}\label{l.fracdom}
The fractional domain spaces associated to $A^E$ are given, for $0<\theta < \frac12$ by 
$E_\theta = W^{2\theta, q}_0(0,1)$ and $E_\theta = W^{2\theta, q}(0,1)$ depending on whether 
the boundary conditions are Dirichlet or Neumann. 
\end{lem}

\begin{proof}
By \cite{daners00}, the semigroup associated to $A^E$ satisfies Gaussian upper bounds. 
Hence, by the 
results of \cite{dr86}, (a suitable translate of) $A^E$ has a bounded $H^\infty$-calculus and 
thus, in 
particular, bounded imaginary powers. Therefore, see \cite[Theorem 6.6.9]{haase}, 
we have $\mathsf{D}((w'-A)^\theta) = [E, \mathsf{D}(A)]_\theta$. 
 
Next observe that by \cite[Theorem 7.1]{amn97} the fractional domain space $E_{\half}$ is given as 
$\mathsf{D}(\mathfrak{a})\cap W^{1,q}(0,1)$, i.e. $W^{1,q}_0(0,1)$ or $W^{1,q}(0,1)$, depending 
on the boundary conditions. Employing \cite[Theorem 6.6.9]{haase} a second time, it follows that
for $\theta\in (0,\half)$ we have
$\mathsf{D}((w'-A)^\theta ) = [E, \mathsf{D}(w'-A)^\half)]_{2\theta}$.
\end{proof}

In either case, the Sobolev embedding implies that $E_\theta$ is
continuously embedded in $B$ for $2\theta > \frac{1}{q}$. Hence the 
conditions (A1), (A4) and (A5) are satisfied 
for $\theta \in (\frac{1}{2q}, \frac{1}{2})$.

The nonlinearity $F$ is modeled as in the previous section, where it was seen  
that (F$'$) and (F$''$) hold. Concerning $G$, we first pick 
$\kappa_G \in (\frac{1}{4}, \frac{1}{2})$. Following \cite[Section 10.2]{vNVW08}
we define the multiplication operator 
$\Gamma : [0,T]\times B \to \cL (H)$ by
\[ [\Gamma (t,u)h](s) := g(t,s, u(s))h(s), \quad s\in [0,1], \]
and then define $G: [0,T]\times B \to \gamma (H, E_{-\kappa_G})$ by
\[ (-A)^{-\kappa_G} G(t,u)h := \iota (-A)^{-\kappa_G}\Gamma (t,u) h \, ,\]
where $\iota : H^{2\kappa_G}(0,1) \to L^q(0,1) = E$.
As in \cite[Section 10.2]{vNVW08} one sees that $G$ takes values
in $\gamma (H, E_{-\kappa_G})$ is locally Lipschitz continuous
as a map from $[0,T]\times B\to \gamma (H, E_{-\kappa_G})$.
It follows 
easily that $G$ satisfies assumption (G$''$).

From Theorem \ref{t.dissex2} we obtain:

\begin{thm}
Let $\frac{p}{4}>2k+1$, where $k$ is the exponent in the reaction term  \eqref{eq:reactionterm}. 
Under the assumptions above, for every $\xi \in L^p(\Omega, \cF_0, \P; B)$
the solution $\bX$ of equation \eqref{eq.spde} exists globally and belongs to 
$L^p(\Omega;C([0,T];B))$. 
\end{thm}
\begin{proof}
The condition $\frac{p}{4}>N$, with $N= 2k+1$, allows us to choose $2\le q<\infty$,
 $\theta\in [0,\frac12)$ and $\kappa_G\in (\frac14,\frac12)$ 
such that, with $E = L^q(0,1)$, we have 
$E_\theta \embed B$ and 
$0\le \theta+\kappa_G < \frac12 - \frac{1}{pN}$.
By the above discussion, the assumptions of Theorem \ref{t.dissex2} are then satisfied. 
\end{proof}

Let us end this article by discussing the dependence of the solution upon 
the coefficients $A, F$ and $G$. Suppose
for every $n \in \OCN$ we are given an operator $\A_n$, determined through its coefficients 
$a_n, b_n$ and $c_n$, and functions $f_n, g_n : [0,T] \times B \to B$. 
Let $A_n, F_n$ and $G_n$ be defined by replacing $\A, f$ and $g$ 
with $\A_n, f_n$ and $g_n$, respectively.

Then we have $F_n (t, u) \to F (t, u)$ in $B$, if $f_n (t, \cdot, \cdot ) \to f (t, \cdot, \cdot )$
for all $t \in [0,1]$, 
uniformly on compact subsets of $[0,1]\times \CR$.
This is a stronger assumption than 
in \cite{KvN10}, where only pointwise convergence was required.
 However, for reaction diffusion equations 
we need convergence in $C([0,1])$.

To infer convergence $G_n(t,u) \to G (t,u)$ for all $t \in [0,1], u \in B$, it is sufficient to have 
convergence $g_n(t,x,s) \to g (t,x,s)$ for all $(t,x,s) \in [0,T]\times [0,1]\times \CR$. Indeed, under 
this assumption we clearly have $\Gamma_n (t,u)h \to \Gamma (t,u) h$ in 
$L^2(0,1)$ for all $t \in [0,T]$ and $u \in B$. 
Hence, by `convergence by right multiplication', see \cite[Proposition 2.4]{vNVW07}, convergence of 
$G_n(t,u) \to G(t,u)$ in $\gamma (H, E_{-\kappa_G})$ follows.

Finally, let us address conditions (A1) -- (A3). If the coefficients $a_n, b_n$ and $c_n$ are uniformly bounded 
and elliptic then the forms $\mathfrak{a}_n$ are uniformly sectorial. This immediately yields (A1), cf.\ 
\cite[Lemma 5.1]{KvN10}. If we have $\lim_n a_n = a, \lim_n b_n = b$ and $\lim_n c_n
= c$ almost everywhere, then (A2) follows from the results of \cite{daners08}. 
Lemma \ref{l.fracdom} shows that the fractional domain spaces $E_{n,\theta}$ coincide as sets and the norms 
are mutually equivalent.
Inspecting  the proofs of the results used in the proof of Lemma \ref{l.fracdom}, the reader may check 
that, in fact, the uniformity assumptions yield that the constants may be chosen independently of $n$. 
This yields (A3).

\appendix 

\section{Convergence of analytic semigroups}

In this appendix we prove some convergence results for analytic semigroups 
under assumptions (A1) -- (A3). 
The lemmas \ref{l.fracconv} and \ref{l.Bconv}
may be known to specialists, but since we could not find these results in the literature 
we include proofs for reasons of completeness. 

The first lemma is used in the proof of Proposition \ref{prop:global}.

\begin{lem}\label{l.fracconv}
Assume (A1) -- (A3). Then
\begin{enumerate}
 \item For all $0\leq \theta < \half$ and $x \in E_\theta$ we have $S_n(t)x \to S_\infty (t)x$ in $E_\theta$, uniformly 
on compact time intervals in $[0, \infty )$.
 \item For all $0\le \theta, \kappa <\half $ and $x \in E_{-\kappa}$ we have 
$A_n S_n(t)x \to A_\infty S_\infty (t)x$ in $E_\theta$, uniformly on compact time intervals in $(0, \infty )$.
\item Let $\theta \in (0, \half )$ and $\lambda, \delta \geq 0$ satisfy $\lambda + \delta < \theta$. 
If $x_n \to x_\infty$ in $E_\theta$, then $S_n(\cdot )x_n \to S_\infty (\cdot )x_\infty$ in $C^\lambda ([0,T], E_\delta )$.
\end{enumerate}
\end{lem}

\begin{proof}
For notational convenience, we will assume that $w< 0$ so that we may choose $w^\prime = 0$ in the definition of the 
fractional domain spaces.

(1) Let $T>0$ be given. For $x \in E_\theta$ we have
\begin{equation}\label{eq.est1}
\begin{aligned}
 \|S_n(t)x - S_\infty (t)x\|_\theta & \simeq \|(-A_n)^\theta S_n(t)x -
 (-A_n)^\theta S_\infty(t)x\|_E \\
 & \leq  \|(-A_n)^\theta S_n(t)x - (-A_\infty )^\theta S_\infty(t)x\|_E
\\ &\quad\quad + 
\|(-A_\infty)^\theta S_\infty (t)x - (-A_n)^\theta S_\infty (t)x\|_E\, ,
\end{aligned}
\end{equation}
where the implied constants in the first line may be chosen independently of $n$ by (A3). Now observe that
for $0 \leq t \leq T$,
\begin{equation}\label{eq.est2}
\begin{aligned}
& \|(-A_n)^\theta S_n(t)x - (-A_\infty)^\theta S_\infty(t)x\|_E \\
 &\quad\leq  C_T \|(-A_n)^\theta x - (-A_\infty)^\theta x\|_E
+ \|S_n(t)(-A_\infty)^\theta x - S_\infty(t) (-A_\infty)^\theta x\|_E ,
\end{aligned}
\end{equation}
where $C_T := \sup\{ \|S_n(t)\|_{\cL (E)} \, : \, 0 \leq t \leq T\, , \, n \in \CN\} < \infty$ as a consequence 
of (A1). We note that the last term in \eqref{eq.est2} converges to 0 uniformly for $t \in [0,T]$ by (A1), (A2) 
and the Trotter-Kato theorem.\smallskip 

We now prove that, given a compact set $K \subseteq E_\theta$, we have $(-A_n)^\theta x \to (-A_\infty)^\theta x$ in 
$E$, uniformly for $x \in K$. This proves that also the first term on the right-hand side of \eqref{eq.est2} 
converges to 0, hence the first term on the right-hand side of \eqref{eq.est1} converges to 0. Moreover, 
since $\{S_\infty (t)x\, : \, 0 \leq t \leq T\}$ is compact in $E_\theta$ for all $x \in E_\theta$,
it also follows that the second term on the right-hand side of \eqref{eq.est1} converges to 0, whence the 
proof of (1) is complete.

In view of the uniform boundedness of $(-A_n)^\theta$ as operators in $\cL (E_\theta, E)$, to prove 
the convergence $(-A_n)^\theta x \to (-A_\infty)^\theta x$, uniformly on compact subsets of $E_\theta$, it actually 
suffices to prove strong convergence on a dense subset of $E_\theta$. To that end, 
pick $\eta \in (\theta, \half)$. 
Then $E_\eta$ is a dense subset of $E_\theta$, see \cite[Proposition 3.1.1]{haase}. Moreover, for $x \in E_\eta$ 
we have $(-A_n)^\theta x = (-A_n)^{\theta - \eta}(-A_n)^{\eta}x$, hence, by \cite[Corollary 3.3.6]{haase},
\[ (-A_n)^\theta x = \frac{1}{\Gamma (\eta - \theta )}\int_0^\infty t^{\eta - \theta -1} (-A_n)^\eta S_n(t)x \, dt\, .\]

Now note that 
\[ \|t^{\eta - \theta - 1}S_n(t)(-A_n)^\eta x\|_E \leq t^{\eta -\theta -1} Me^{wt}\sup_{n\in\CN}\|(-A_n)^\eta\|_{\cL (
 E_\eta, E)}\|x\|_{\eta}\, ,
\]
which is certainly integrable on $(0,\infty )$. Moreover, $(-A_n)^\eta S_n(t)x \to (-A_\infty)^\eta S_\infty(t)x$ 
for all $t \in (0,\infty )$ which, using (A1) and (A2), is easy to see by employing dominated convergence in 
a contour integral representation for $(-A_n)^\eta S_n(t)$.

Thus, by dominated convergence, $(-A_n)^\theta x$ converges in $E$ to $(-A_\infty)^\theta x$, for all $x \in E_\eta$.
This finishes the proof of (1).\medskip 

(2)  Fix $0<\eps < T$. We have
\[
\begin{aligned} 
 \|A_n S_n(t) - A_\infty S_\infty (t)x\|_\theta  & \simeq \|(-A_n)^{\theta}A_n 
S_n(t)x - (-A_n)^\theta A_\infty S_\infty (t)x\|_E\\
& \leq \| (-A_\infty)^{\theta+1} S_\infty(t)x - (-A_n)^{\theta +1} S_n (t)x \|\\
& \quad \quad  + \|(-A_\infty )^\theta A_\infty  
S_\infty (t)x - (-A_n)^\theta A_\infty  S_\infty (t)x\|_E\, .
\end{aligned}
\]
Convergence of the first term to 0, uniformly on $[\eps, T]$, can be proved by a contour integral argument 
in the extrapolation space $E_{-\kappa}$. Convergence of the second 
term follows from the convergence of $(-A_n)^\theta x \to (-A_\infty )^\theta x$, uniformly on the 
set $\{ A_\infty S_\infty (t)x\, : \, \eps \leq t \leq T\}$, which is a compact subset of $E_\theta$.\medskip 

(3) Pick $\eps >0$ such that $\lambda + \delta + \eps < \theta$. Then, for $t,s \in [0,T]$, we have
\[
 \begin{aligned}
 \|S_n(t)x_n - S_n(s)x_n\|_\delta & \, \simeq \| (-A_n)^\delta S_n(t)x_n - (-A_n)^\delta S_n(s)x_n\|_E\\
& \leq C (t-s)^{\lambda + \delta +\eps} \|(-A_n)^{\lambda + \delta + \eps} x_n \|_E \lesssim 
C (t-s)^{\lambda + \delta + \eps} \|x_n\|_\theta\, ,
 \end{aligned}
\]
where $C$ is a constant only depending on $M$ and $w$ in (A1). Furthermore, the implied constants in 
the first and the last step can be chosen independently of $n$. Since $x_n$ is convergent, hence bounded, in 
$C_\theta$, it follows that the sequence $(S_n(\cdot )x_n)_{n\in\CN}$ is bounded in $C^{\lambda +\eps}([0,T], E_\delta )$. 
Moreover, by (1), the continuity of the embedding $E_\theta \inject E_\delta$ and the uniform boundedness of 
$\bS_n$ on $E_\theta$, it follows that $S_n(\cdot )x_n \to S_\infty (\cdot )x_\infty$ in $C([0,T]; E_\delta )$. 
This clearly yields that $S_n(\cdot )x \to S_\infty (\cdot )x_\infty$ in $C^\lambda ([0,T], E_\delta )$.
\end{proof}

If, in addition, (A4) holds, we have the following result.

\begin{lem}\label{l.Bconv}
Assume (A1) -- (A4).
For all  $0\le \theta <\half$ and $x \in B$ we have
$S_n(\cdot )x \to S(\cdot )x$ in $C([0,T]; B)$.
\end{lem}

\begin{proof}
By Lemma \ref{l.fracconv}, we have $S_n(\cdot )x \to S(\cdot)x$ in $C([0,T]; E_\theta) \inject
C([0,T]; B)$ for all $x \in E_\theta$. By the density of $E_\theta$ in $B$ and 
the uniform 
exponential
boundedness of $\bS_n^B$, this extends to all $x \in B$.
\end{proof}

\bibliographystyle{plain}
\bibliography{reaction-diffusion}

\end{document}